\documentclass[11pt]{amsart}  
\usepackage{amssymb,amsmath,amsthm,amscd}
\usepackage{graphicx}
\numberwithin{equation}{section}

\newtheoremstyle{fancy1}{10pt}{10pt}{\itshape}{12pt}{\textsc\bgroup}{.\egroup}{8pt}{
}
\newtheoremstyle{fancy2}{10pt}{10pt}{}{12pt}{\itshape}{.}{8pt}{ }

\theoremstyle{fancy1}

\newtheorem{lem}[equation]{Lemma}

\newtheorem{thm}[equation]{Theorem}

\newtheorem{main}{Theorem}
\newtheorem*{main*}{Theorem}

\newtheorem*{cor*}{Corollary}
\newtheorem*{prop*}{Proposition}

\newtheorem*{problem*}{Problem}

\theoremstyle{fancy2}

\newtheorem{rem}[equation]{Remark}
\newtheorem*{rem*}{Remark}

\newcommand{\cref}[1]{Corollary~\ref{#1}}

\newcommand{\lref}[1]{Lemma~\ref{#1}}

\newcommand{\rref}[1]{Remark~\ref{#1}}
\newcommand{\tref}[1]{Theorem~\ref{#1}}

\newcommand{\gt}{\theta}
\newcommand{\e}{\epsilon}

\newcommand{\RP}{\mathbb{R\mkern1mu P}}
\newcommand{\CP}{\mathbb{C\mkern1mu P}}

\newcommand{\Sph}{\mathbb{S}}
\newcommand{\Disc}{\mathbb{D}}

\newcommand{\C}{{\mathbb{C}}}
\newcommand{\R}{{\mathbb{R}}}
\newcommand{\Z}{{\mathbb{Z}}}

\newcommand{\SO}{\ensuremath{\operatorname{SO}}}

\newcommand{\U}{\ensuremath{\operatorname{U}}}
\newcommand{\SU}{\ensuremath{\operatorname{SU}}}

\renewcommand{\S}{\ensuremath{\operatorname{S}}}

\newcommand{\fg}{{\mathfrak{g}}}
\newcommand{\fk}{{\mathfrak{k}}}
\newcommand{\fh}{{\mathfrak{h}}}
\newcommand{\fm}{{\mathfrak{m}}}
\newcommand{\fn}{{\mathfrak{n}}}
\newcommand{\fa}{{\mathfrak{a}}}

\newcommand{\fp}{{\mathfrak{p}}}

\newcommand{\fl}{{\mathfrak{l}}}

\def\con#1=#2(#3){#1 \equiv #2 \bmod{#3}}

\newcommand{\ml}{\langle}                     
\newcommand{\mr}{\rangle}                    

\newcommand{\Ad}{\ensuremath{\operatorname{Ad}}}

\renewcommand{\sec}{\ensuremath{\operatorname{sec}}}

\DeclareMathOperator{\spam}{span}

\newcommand{\Kpm}{K^{\scriptscriptstyle{\pm}}}
\newcommand{\Kp}{K^{\scriptscriptstyle{+}}}
\newcommand{\Km}{K^{\scriptscriptstyle{-}}}
\newcommand{\Ko}{K_{\scriptscriptstyle{0}}}

\newcommand{\no}{\noindent}
\newcommand{\co}{{cohomogeneity}}
\newcommand{\coo}{{cohomogeneity one}}
\newcommand{\com}{{cohomogeneity one manifold}}

\newcommand{\pc}{positive curvature}

\newcommand{\hR}{\hat{R}}

\title{An exotic T$_1$S$^4$ with positive curvature}

\author{Karsten Grove}
\address{University of Notre Dame}
\email{kgrove2@nd.edu}
\author{Luigi Verdiani}
\address{University of Firenze}
\email{verdiani@math.unifi.it}
\author{Wolfgang Ziller}
\address{University of Pennsylvania}
\email{wziller@math.upenn.edu}
\thanks{ The first named author was supported in part by the
Danish Research Council and by a grant from the National Science
Foundation. The second named author was supported by GNSAGA. The
third named author was supported by a grant from the National
Science Foundation, and by  CNPq-Brazil. }

\begin{document}



\begin{abstract}
We construct a metric  with positive sectional curvature on a 7-manifold
which supports an isometry group with orbits of codimension 1.
It is a connection metric on the total space of an orbifold 3-sphere
bundle over an orbifold 4-sphere. By a result of S.~Goette, the manifold
is homeomorphic but not diffeomorphic to the unit tangent bundle of the
4-sphere.
\end{abstract}

\maketitle


Spaces of positive curvature play a special role in geometry.
Although the class of manifolds with positive (sectional) curvature
is expected to be relatively small, so far there are only a few
known obstructions. Moreover, for closed simply connected manifolds
these coincide with the known  obstructions to nonnegative curvature
which are: (1) the Betti number theorem of Gromov which asserts that
the homology of a compact manifold with non-negative sectional
curvature has an a priori bound on the number of generators
depending only on the dimension, and (2) a result of Lichnerowicz
and Hitchin implying that a spin manifold with non-trivial $\hat{A}$
genus or generalized $\fa$ genus cannot admit a metric with non
negative curvature.

\smallskip

One way to gain further insight is to construct and analyze
examples. This is quite difficult and has been achieved only a few
times. Aside from the classical rank one symmetric spaces, i.e., the
spheres and the projective spaces with their canonical metrics, and
the recently proposed deformation of the so-called Gromoll-Meyer
sphere \cite{PW2}, examples were only found in the 60's by Berger
\cite{Be}, in the 70's by Wallach \cite{Wa} and by Aloff and Wallach
\cite{AW}, in the 80's by Eschenburg \cite{Es1,Es2}, and in the 90's
by Bazaikin \cite{Ba}. The examples by Berger, Wallach and
Aloff-Wallach were shown, by Wallach in even dimensions \cite{Wa}
and by Berard-Bergery \cite{Bb} in odd dimensions, to constitute a
classification of simply connected homogeneous manifolds of positive
curvature, whereas the examples due to Eschenburg and Bazaikin
typically are non-homogeneous, even up to homotopy. All of these
examples can be obtained as quotients of compact Lie groups $G$ with
a biinvariant metric by a free isometric ``two sided" action of a
subgroup $H \subset G \times G$. Since a Lie group with a
biinvariant metric has nonnegative curvature so do such quotients,
and in rare cases one even gets positive curvature. To achieve this
no further curvature computations are required, it suffices to show
that any horizontal 2-plane, when translated back to the identity in
$G$, cannot contain two vectors whose Lie bracket is $0$. See
\cite{Z1} for a survey of the known examples.

\smallskip

Our main purpose here is to present a new method for the construction of
positively curved manifolds, and to use it to exhibit one new manifold with
 positive curvature (see \cite{De2} for an independent and different approach):

\begin{main}
There is a positively curved 7-manifold, which is homeomorphic, but not diffeomorphic to the unit tangent bundle of the 4-sphere.
\end{main}

Our result is actually stronger than stated: We exhibit an explicit metric $g$ and a 4-form $\eta$, and prove that the modified curvature operator $\hat R + \hat \eta$ on the bundle of 2-forms (automatically having the same ``sectional curvatures") is positive.  Recall that $\hat R$ itself being positive is extremely strong and only can happen for manifolds diffeomorphic to space forms \cite{BW}. The idea to consider such modified curvature operators was pioneered by Thorpe in dimension 4, and
implemented in higher dimensions by P\"{u}ttmann \cite{putmann}, where it was shown that all homogeneous positively curved metrics have \emph{strongly positive curvature} in this sense. It is the first time, however, that this method has been used to establish positivity of curvature in a new example.

The example is indeed a new one, since  $T_1\Sph^4$ is 2-connected
with third homotopy group $\Z_2$, and the only other known
2-connected positively curved  7-manifolds are  $\Sph^7$ and the
Berger space $B^7 = \SO(5)/\SO(3)$ with $\pi_3(B^7)=\Z_{10}$
\cite{Be}. It is a highly non-trivial and recent result due to S.
Goette \cite{G} that our new example is diffeomorphic to a
3-sphere bundle over the 4-sphere, and is homeomorphic but not
diffeomorphic to $T_1\Sph^4$ (see \cite{CE},\cite{KS}, and
\cite{Cr} for a proof that they are homeomorphic). Furthermore, from
\cite{KZ},\cite{To}, it follows that it is not diffeomorphic to any
biquotient. We point out that it is not yet known if $T_1\Sph^4$
itself has a metric of positive curvature, but P.Petersen and
F.Wilhelm have shown that it supports a metric with positive
curvature on an open and dense set \cite{PW1}.

\smallskip

Our example, is the second among an explicitly given infinite sequence $\{P_k\}, k = 1, 2, \ldots$ of 2-connected cohomogeneity one $\SO(4)$ 7- manifolds with $\pi_3(P_k) = \Z_k$ for which no obstructions to positive curvature are known (cf. \cite{GWZ}), the first being $P_1 = \Sph^7$.
By construction, the subaction by $\S^3 \subset \SO(4)$ on $P_k$ also yields the structure of an \emph{orbifold} principle $\S^3$- bundle over $\Sph^4$, and our metric on $P_2$ is an \emph{orbifold connection metric} for this bundle. Here the orbifold setting is crucial, since it is well known that a connection metric on a smooth
$\Sph^3$ bundle over $\Sph^4$ has positive curvature only in the
case of the Hopf bundle, where the total space is $\Sph^7$,
\cite{DR}.

\smallskip

In general, the attempt to describe and eventually classify
positively curved manifolds with large isometry group provides a
natural framework for a systematic search for new examples ( see
\cite{grove:survey},\cite{Wi2}). The manifold $P_2$ has indeed
emerged in this context: Specifically, in
\cite{verdiani:1,verdiani:2} and \cite{GWZ}  an exhaustive
description was given of all simply connected cohomogeneity one
manifolds that can possibly support an invariant metric with
positive curvature. In addition to the normal homogeneous manifolds
of positive curvature and a subset among the Eschenburg and Bazaikin
spaces which admit a cohomogeneity one action, two infinite
families, $P_{k}, Q_{k}$ and one exceptional manifold $R$, all of
dimension seven (with ineffective actions of $\S^3\times\S^3$),
appeared as the only possible new candidates (see \cite{GWZ} and the
survey \cite{Z2}). Here $Q_1$ is the normal homogeneous positively
curved Aloff-Wallach space (\cite{Wi1}). Recently it was shown  in
\cite{VZ2} that the exceptional candidate $R$ in fact does not admit
an invariant metric with positive curvature.

\smallskip

 It is a curious fact, as was proved in \cite{GWZ}, that the infinite families
admit a different description: They are the two-fold universal
covers, $P_{k} \to H_{2k-1}$ and $Q_{k} \to H_{2k}$ of the frame
bundle $H_{\ell}$ of self-dual 2-forms associated to the self dual
Einstein orbifolds $O_{\ell}$ ($\Sph^4$ with an $\SO(3)$ invariant orbifold metric) constructed by Hitchin in \cite{Hi1}.
As such, these manifolds come with natural 3-Sasakian metrics, that in particular are (orbifold) connection metrics. There is a general
necessary and sufficient condition for a connection metric to have positive curvature once the fiber is shrunk sufficiently (see \cite{CDR}), that also applies in the orbifold context. In the special case of 3-Sasakian metrics this is equivalent to the base having positive curvature.
Unfortunately the curvature of the Hitchin metrics are positive only
for $O_1$ and $O_2$.
However, on $O_3$ (the base of
$P_2$) this metric has positive curvature on a large region and only
relatively small negative curvature , see Figure 8 in \cite{Z2}.
This suggests that it might be possible to make a small change of
the Hitchin metric on $O_3$ with positive curvature, choose a
principal connection close to the Hitchin connection, and get
positive curvature on the total space after shrinking the metric on
the fiber sufficiently.
We use the Hitchin metric and connection as a guide only.  Our metric on the base, and the
principal connection, are explicitly given by polynomials. For this
we divide the interval on which the metric is defined into three
subintervals, two close to the singular orbits, and a larger one in
the middle. Near the singular orbits we find functions  consisting
of polynomials of degree  $3$. In the middle we glue with the unique
polynomials of degree  $5$ such that the resulting metric on the
manifold is $C^2$ (See \eqref{vfunctions} and \eqref{hfunctions} for
the explicit formulas). It is then obvious that any smooth $C^2$
perturbation will have positive curvature as well. To prove that our metric has positive curvature (on each piece), the crucial and non-trivial point is to find and
add an invariant 4-form so as to make the modified curvature operator positive
definite when the fiber metric is shrunk sufficiently. To prove
positive definiteness, given our choices, boils down to checking
that specific polynomials with integer coefficients have no zeroes
on a particular closed interval. This is done by using Sturm's
theorem, which counts real zeroes of such polynomials by computing
the $\gcd$ of the polynomial and its derivative (i.e. applying the
Euclidean algorithm).

The metric by O.Dearricott in \cite{De2} differs from ours in that
he deforms the self dual Hitchin metric  on the base of the orbifold
bundle conformally, but keeps the principal connection as the one
coming from the Hitchin metric.

\smallskip

It is a natural conjecture that all the manifolds $P_{k}$ and
$Q_{k}$ admit invariant metrics of positive curvature. This would be
particularly interesting for the $P_{k}$ family, since they are all
2-connected, hence contradicting a conjecture in
\cite{fang-rong:finiteness}. This requires a more drastic change of
the Hitchin metrics on the base and hence difficulty in a natural
choice of principal connection using our method. It is not difficult
to construct invariant metrics of positive curvature on the base
(using, e.g., Cheeger deformations), but corresponding choices of
principal  connections will require new insights. We point out that
the class of connection metrics, while simpler to work with
geometrically, is considerably smaller than the class of general
invariant metrics. In particular, we will show that the
manifold $B^7$, although it admits a cohomogeneity one metric with positive curvature, does not  admit a connection metric with positive curvature.

\smallskip

Here is a short description of the individual sections. In Section 1
we describe the $P_k$ as well as the $Q_k$ families
 including all invariant metrics on them in terms
 of functions on the orbit space interval. The imposed boundary
 conditions for these functions and general curvature formulas are easily obtained and described in the Appendix. Section 2 is
 devoted to a discussion of connection metrics in our context
 and the corresponding simplified curvature formulas and
  smoothness conditions.  A discussion of the Thorpe
    method and how to choose a suitable invariant 4-form is discussed in Section 3, and  Section 4 describes the metric and the principal connection.  The proof  that the constructed metric
        and chosen 4-form has positive definite ``curvature
        operator" is carried out in Section 5.

\smallskip

The present paper is a
minor modification of \cite{GVZ}, first made available on the arXiv  with
a different title.

\smallskip

It is a pleasure to thank Burkhard Wilking for helpful discussions
and Peter Storm for suggesting the use of Sturm's theorem in our
proof. The second and third named author were also supported
   by IMPA in Rio de Janeiro
    and would like
   to thank the Institute for its hospitality.

\bigskip

\section{Candidates and their invariant metrics}

\bigskip

To establish notation, we begin with a brief review of the basic description of cohomogeneity one manifolds and their invariant metrics (for more details, we refer
to  \cite{AA,GZ1,GWZ}).

Let $G$ be a compact Lie group which acts isometrically on a compact
Riemannian manifold $M$ with orbit space an interval. The interior
points of the interval correspond to the principal orbits, and the
end points to the non-principal orbits $B_\pm$ (singular in the case
of simply connected $M$). Let $c: [0,L] \to M$ be a distance
minimizing geodesic parameterized by arclength connecting the
non-principal orbits. The isotropy group at $c(0)$ is denoted by
$\Km$ and the one  at $c(L)$ by $\Kp$. The  principal isotropy
group, constant for $c(t), 0<t<L$, is denoted by $H$. Since the
boundary of tubes around the singular orbits must be regular orbits,
we have that $\Kpm/H$ are spheres.

An important property of \com s is that a converse also holds:  If
we have compact groups with inclusions $H\subset \{\Km ,
\Kp\}\subset G$ satisfying $\Kpm/H=\Sph^{\ell_\pm}$, then one can
define a \com\ by gluing the two disc bundles
$G\times_{\Km}\Disc^{\ell_{-}+1}$ and
$G\times_{\Kp}\Disc^{\ell_{+}+1}$ along their common boundary $G/H$
via the identity. One possible description of our manifold is thus
simply in terms of the \emph{diagram} of groups $H\subset \{\Km ,
\Kp\}\subset G$.

\smallskip

To describe a $G$ invariant metric on $M$, it suffices to
describe the metric along  $c$. For $0<t<L$,
 $c(t)$ is a regular
 point  with constant isotropy group $H$ and the metric on the
 principal
 orbits $G c(t)=G/H$ is a smooth family of homogeneous metrics
 $g_t$.
 Thus on the regular part
 the metric $\langle\ ,\ \rangle _{c(t)} = g_{c(t)}$ is determined by
$$g_{c(t)}=d\,t^2+g_t,$$
and since the regular points are dense it also describes the metric
on $M$. In terms of a fixed biinvariant inner product $Q$ on the Lie algebra $\fg$ and corresponding $Q$-orthogonal splitting $\fg=\fh\oplus\fm$ we have $\Ad(H)(\fm)\subset\fm$ and the tangent space to
$G/H$ at $c(t), t\in(0,L)$ is identified with $\fm$ via action fields: $X\in\fm\to
X^*(c(t))$. With this terminology the metric  $g_t$ is an $\Ad(H)$-invariant inner product on $\fm$. In terms of $Q$ we also have the representation
$$g_t(X^*,Y^*) = Q(P_t(X),Y)$$
 where $P_t: \fm \to \fm$ is a positive, symmetric $\Ad(H)$
 equivariant operator for each $t \in (0,L)$.
 When extended to the closed interval $0 \le t \le L$,
 $g_t$ degenerates at the end points, and smoothness
 of the metric on $M$, correspond to explicit
 \emph{boundary conditions} for $g_t$ at $0$ and
 at $L$ imposed by invariance (cf. \cite{BH},\cite{EW}).

\bigskip

We will now recall the explicit description of our specific
candidates from \cite{GWZ} in terms of group diagrams as above, and
use it to describe all smooth invariant metrics on them.

\bigskip

\begin{center}
 \emph{Metrics on the $P$ family.}
 \end{center}

\bigskip

 Regarding $\S^3$ as the unit quaternions,
 the group diagram for $P_k$ is given by:
\begin{equation}\label{Pgroups}
H=\Delta Q \subset  \{ (e^{i\gt},e^{i\gt})\cdot H \, ,
(e^{j(1+2k)\gt},e^{j(1-2k)\gt})\cdot H\}\subset \S^3\times \S^3,
\end{equation}
where $H$ is isomorphic to the quaternion group $Q=\{ \pm1 , \pm i,
\pm j, \pm k\}$, embedded diagonally in $\S^3\times \S^3$. The signs
of these slopes differ from the ones in \cite{GWZ}, but do not
affect the equivariant diffeomorphism type since we can conjugate
all groups by $(1,i)$. This change simplifies the smoothness conditions.

\bigskip

Since $H$ is finite in our case, $\fm= \fh^{\perp} = \fg$ with the notation above. For a basis of $\fg$ we let
 $X_i$ and $Y_i$ be the  left invariant vector fields on $\S^3 \times
\S^3$
 corresponding to $i,j$ and $k$ in the Lie
 algebras
of the first and second $\S^3$ factor of $G$. The adjoint action of
$H$ is in our case by sign changes in the basis vectors $X_i,Y_i$.
For example, $\Ad(i,i)$ fixes  $X_1$ and $Y_1$ and multiplies $X_2,
Y_2 , X_3, Y_3$ by $-1$, and similarly for $(j,j), (k,k)\in H$. This
implies in particular that
$$ \langle X^*_i,X^*_j \rangle  =  \langle X^*_i,Y^*_j \rangle
  =  \langle Y^*_i, Y^*_j \rangle  = 0 \text{ for all } i \ne
j$$
The metric is therefore described by 9 functions:
\begin{equation}\label{functions}
f_i(t) = \  \langle X^*_i,X^*_i \rangle_{c(t)}\quad,\quad
 g_i(t) = \   \langle Y^*_i, Y^*_i \rangle_{c(t)}
\quad,\quad h_i(t)= \langle X^*_i,Y^*_i \rangle_{c(t)}
\end{equation}
 all defined on $[0,L]$.

 \bigskip

\bigskip

\begin{center}
 \emph{Metrics on the $Q$ family.}
 \end{center}

\bigskip

For completeness, we now shortly discuss the second family of
candidates. The group diagram for $Q_k$ is given by
\begin{equation}\label{Qgroups}
H\simeq \Z_2\oplus\Z_4 =\{(\pm 1, \pm 1), (\pm i, \pm i)\}\subset \{
(e^{i\gt},e^{i\gt})\cdot H \, , (e^{j(k+1)\gt},e^{-jk\gt})\cdot
H\}\subset \S^3\times \S^3.
\end{equation}

We point out that in this case, since $H$ is smaller than for
$P_k$, a general \coo\ metric can have other non-zero inner products
of the basis vectors $X_i,Y_i$, but the curvature formulas in the
general case are significantly more complicated.
Moreover, metrics lifted from the corresponding
 $\Z_2$ quotients $H_{\ell}$  are of the above form.

\bigskip

Additional strong restrictions on the functions defining the metric
on $P_k$ or $Q_k$ are imposed at the end points of the interval
$[0,L]$, where the principal orbits collapse
 and is carried out for our candidates in the Appendix (see Theorem \ref{gen smooth}).
The curvature tensor of a general cohomogeneity one metric on $P_k$
and $Q_k$ is easily obtained from known formulas and is discussed in the
Appendix as well (see Theorem \ref{gencurv}).

\bigskip

\section{Connection Metrics}

\bigskip

We now restrict our general type of \coo\ metrics to so-called
connection metrics. This will simplify the curvature formulas
significantly (in particular when the vertical part of the metric is
scaled by $\epsilon$), but also  enables one to understand the
behavior of the functions in a more geometric fashion.

\bigskip

In general, when $G$ contains a normal subgroup $L \triangleleft G$ which acts freely (or almost freely) on $M$ the quotient map $\pi: M \to M/L$ is a principal (orbifold) $L$ bundle over the $G/L$ cohomogeneity one (orbifold) base $B = M/L$. In this case, a subfamily of invariant metrics are \emph{connection metrics}, i.e., metrics of the form

\begin{equation}\label{connectionmetric}
 \langle U,V \rangle  = g_B(\pi_*(U), \pi_*(V)) +
  Q(\theta(U), \theta(V))
\end{equation}

\no where  $g_B$ is a ($G/L$ invariant) metric on the base
 $B$,  $ Q$ a  bi-invariant metric on $\fl$, and $\theta$
 a (G invariant) connection, i.e., an invariant
  choice of a complement to the tangent spaces of the
   $L$-orbits. Note that one gets a natural family of
   such metrics simply by scaling $Q$ by $\epsilon$.

\smallskip

In our case, the above discussion applies to
$L=\S^3\times\{1\}\subset \S^3\times\S^3 = G$. Indeed, since $L$ is
a normal subgroup of $G$, the isotropy groups of $L$ are simply the
intersection of $L$ with the isotropy groups along $c(t)$. For $P_k$
the isotropy groups are hence trivial on $B_-$ and the principal
orbits, and along $B_+$
 equal to $\{e^{j(1+2k)\gt}\mid e^{j(1-2k)\gt}=1\}\simeq\Z_{2k-1}$.
We thus have an \emph{orbifold principal} $\S^3$ bundle, $\pi: P_k
\to P_k/\S^3\times\{1\} = B$. The base $B$ carries a \co\ action
induced by $\{1\}\times \S^3$ since it commutes with $L$. Its
isotropy groups are $e^{i\gt}\cdot H\  , \ e^{j\gt}\cdot H\ , $
where $ H=\{\pm 1,\pm i,\pm j,\pm k\}$. This action is $\Z_2$
ineffective, and the corresponding effective action by $\SO(3)$ is
in fact
 the remarkable cohomogeneity one action on $\Sph^4$, whose
 extension to $\R^5$, viewed as the symmetric traceless
  $3\times3$ matrices, is by conjugation, see e.g. \cite{GZ1}.
   Each of the two singular orbits are the
    so-called Veronese surfaces $\RP^2 \subset \Sph^4$.
      The metric on $B$ is smooth except along the right
       singular orbit $B_+/\S^3 = \RP^2$ where the
        ``normal bundle" has fibers that are Euclidean
        cones over circles of length $2\pi/(2k-1)$.

\smallskip

Similarly, from the isotropy groups of the $\S^3\times\S^3$ action
on $Q_k$ in \eqref{Qgroups}, it again follows that
$\S^3\times\{1\}\subset \S^3\times\S^3 $ acts almost freely with
isotropy groups along $B_+$ equal to $\{e^{j(k+1)\gt}\mid
e^{jk\gt}=1\}\simeq\Z_k$, and trivial otherwise. In this case, the
base $Q_k/\S^3\times\{1\}$ has an induced action by $\S^3$ with
isotropy groups  $e^{i\gt}\cdot H\ , \ e^{j\gt}\cdot H\ , $ where $
H=\{\pm 1,\pm i\}$. This action is again  $\Z_2$ ineffective and the
induced action by $\SO(3)$ has the same isotropy groups as the
action of $\SO(3)\subset\SU(3)$ on $\CP^2$, see e.g. \cite{Z2}. The
metric on the base $\CP^2$ is also smooth everywhere in this
 case, except along the right singular orbit where the normal
 spaces are cones on circles of length $2\pi/k$.

\smallskip

To make the discussion of $P_k$ and $Q_k$ more uniform, we can
further compose the projection $\pi: Q_k \to Q_k/\S^3\times\{1\}
=\CP^2$ with the two fold branched cover $\CP^2\to\Sph^4$ obtained
orbitwise from the respective $\SO(3)$ actions. From the above
description of the  isotropy groups of these actions one sees that
this is a 2-fold cover along the principal orbits and the left hand
side singular orbit. But along the right hand side singular orbit it
is a diffeomorphism, which can thus be considered to be the
branching locus. Orthogonal to this singular orbit it divides
angles by $2$. Thus we can also regard $Q_k$ as an orbifold
principal bundle over $\Sph^4$ with angle normal to $B_+$ equal to
$2\pi/(2k)$. As we will see shortly, we will then be able  to deal
with $P_k$ and $Q_k$ at the same time.

\bigskip

We now claim that the  cohomogeneity one metrics from Section 1
with
$$f = f_i: = \  \langle X^*_i,X^*_i \rangle
\quad \text{all \emph{constant} and equal    }
$$
in fact are connection metrics for the principal $L$-bundle
$\pi\colon M\to B$. Indeed, the horizontal space is invariant under
$L$ by definition. For inner products along orbits we have $ \langle
A^*, B^* \rangle _{g \pi(c(t))} = \langle \Ad(g^{-1}) A, \Ad(g^{-1})
B \rangle _{\pi(c(t))}$ for $A,B\in\fg$.
 Since $f_i = f$ is constant
and the $\{1\} \times\S^3$ action commutes with the $\S^3 \times
\{1\}$ action we see that   $ \langle X^*_i, X^*_j \rangle $ are
constant along $\S^3 \times \S^3$ orbits as well as along $c$. In
particular, the $X^*_i$ are orthogonal everywhere, with constant
length $\sqrt f$. Setting $f=\e$, the metric is a connection metric
as in \eqref{connectionmetric} scaled by $\epsilon$.

\smallskip

 The \emph{vertical} space $\mathcal{V}$ at any point of an
$\S^3\times \{1\}$ orbit is spanned by the $X^*_i$, i.e.,
$$\mathcal{V} = \spam\{X^*_i\}$$
For the \emph{horizontal} space $\mathcal{H}$  we thus have:
$$\mathcal{H} = \spam\{T, V_i\}, \ \text{where} \ T:=c'(t) \ \text{and} \  V_i :=
Y^*_i - \sum_j \frac{1}{f}  \langle Y^*_i, X^*_j \rangle  X^*_j$$

Note that since $ \langle Y^*_i, X^*_j \rangle $ are not constant
along either $\S^3$ orbit, the vector fields $V_i$ are not action
fields. But along $c$ the vector fields $V_i$ are orthogonal with:

  $$V_i = Y^*_i - \frac{h_i}{f}X^*_i\ \ ,\ \  \langle
  V_i,V_i \rangle
   = g_i -\frac{h_i^2}{f} = \frac{fg_i -
h_i^2}{f}=: v_i^2$$

The second $\S^3=\{1\}\times\S^3 $  induces an action on the
quotient $B$ and for the induced basis $i,j,k$ we denote the action
fields by $W^*_i$.  Then $V_i$ are the \emph{horizontal lifts} of
$W^*_i$ and hence

$$\langle W^*_i,W^*_i \rangle =v_i^2 \ , \ \langle W^*_i,W^*_j \rangle =0
\ \text{ for } \ i\ne j$$

\smallskip

We now define the unit vectors
$$Z_i=W_i^*/|W_i^*|  \text{ and their horizontal
lifts } \bar{Z}_i=V_i/|V_i|.$$ From now on all curvatures will be
expressed in terms of the unit vectors $T$, $Z_i$ and $\bar{Z}_i$ (and the vectors $X^*_i$).
Notice that these are well defined (along the normal geodesic
$c(t)$) even at the singular orbits and hence all curvature
conditions hold on all of $M$.

\smallskip

The data $(f,h_i,v_i)$ completely describe the metric since we can
recover $g_i$ via $g_i=v_i^2 +h_i^2/f$. We want to {\it scale} the
metric in direction of the fibers by an amount $\e$, keeping the
horizontal space and the metric on the base the same. We claim that
this corresponds to:

\begin{equation}\label{scaling}
(f,h_i,v_i) \to (\e f, \e h_i, v_i) \text{ and }
g_i\to v_i^2+\e\frac{h_i^2}{f}=g_i-(1-\e)\frac{h_i^2}{f}
\end{equation}

Indeed, we then have in the new metric

$$\langle X_i,X_i\rangle = \e f \ , \ \langle X_i,V_i\rangle =\langle X_i,Y_i-\frac{h_i}f X_i\rangle =\e h_i-\frac{h_i}f \e f=0$$
and
$$\langle V_i,V_i\rangle =\langle Y_i-\frac {h_i}f X_i,Y_i-\frac{h_i}f X_i\rangle
=v_i^2+\e\frac{h_i^2}{f} -2\frac{h_i}{f}\e h_i+\frac
{h_i^2}{f^2}\e f =v_i^2$$

Since we want to study conditions for \pc\ under the assumption that
$\e\to 0$, it does not matter where we start, and we will thus set
$f=1$ from now on. Then the metric is described by the functions
$(1,h_i,v_i)$ and is changed to $(\e,\e h_i,v_i)$ under scaling.

\smallskip

In this language, our new example of positive curvature is described
by the formulas in \eqref{vfunctions} for the $v_i$ functions and
\eqref{hfunctions} for the $h_i$ functions.

\bigskip

For the connection form  $\theta$ we have
\begin{equation}\label{connection}
\theta(X^*_i) =  X_i \ \ ,\ \
\theta(V_i) =\theta(T)= 0 \ \text{ and thus }\ \ \theta(Y^*_i) = h_i
X_i \end{equation}
 Thus the functions $h_i$ can be considered to be the principal
connection whereas the $v_i$'s represent the metric on the base.

\bigskip

\begin{center}
 \emph{Smoothness of connection metrics.}
 \end{center}

\bigskip

We now describe the smoothness of the metric in terms of $v_i$ and
$h_i$. For this we unify the description of $P_k$ and $Q_k$, by
regarding each as an orbifold principal bundle over $\Sph^4$ as
above. The metric on the base, which we denote by $O_\ell$, has an
orbifold singularity normal to the Veronese surface with angle
$2\pi/\ell$, as in the case for the Hitchin metric. Thus $\ell=2k-1$
 gives rise to a metric on $P_k$ and $\ell=2k$ one on $Q_k$. The
 metric we construct will only be $C^2$, and the smoothness
 conditions are given by:

\begin{thm}\label{smooth conn1} If $\ell >2$, a
 connection metric, described by the
functions $v_i(t)$ on the base $O_\ell$,  and the principal
connection $h_i(t)$, is $C^2$ if and only if:
\begin{align*}
&v_1(0)=0\ ,\ v_1'(0)=4 \ ,\ v_1''(0)=0 \ , \  v_2(0)=v_3(0)
\  ,  \  v_2'(0)=-v_3'(0)
\  , \  v_2''(0)=v_3''(0)\\
 &  v_2(L)=0\;  , \;  v_2'(L)=-4/\ell \;  ,  \; v_2''(L)=0 \; , \; v_1(L)=v_3(L)
 \; , \;   v_1'(L)=v_3'(L)=0
\;  ,  \; v_1''(L)=v_3''(L)
 \end{align*}
and the principal connection satisfies:
 \begin{align*}
&\qquad h_1(0)=-1\ ,\ h_1'(0)=0 \ ,\   h_2(0)=h_3(0)
\  ,  \  h_2'(0)=-h_3'(0)
\  , \  h_2''(0)=h_3''(0)\\
 &  h_2(L)=\frac{\ell+2}{\ell}\;  , \;  h_2'(L)=0 \;  ,  \; h_1(L)=h_3(L)=0
 \; , \;   h_1'(L)=-h_3'(L)
\;  ,  \; h_1''(L)=h_3''(L)=0
 \end{align*}
\end{thm}

\begin{proof}
Using $ g_i=v_i^2+h_i^2 $ and $f_i=1$, this easily follows from the
general smoothness conditions in \tref{gen smooth}. Notice though
that the ineffective kernel for the action of $\Kpm$ on the normal
sphere is $\Z_4$ for the $P_k$ family  and $\Z_2$ for the $Q_k$
family. Due to this fact, the smoothness conditions take on the same
form.
\end{proof}

{\it Remark.}
A crucial difference between $\ell=1$ and $\ell>1$ in this
language is  that $v_1'(L)=0$ is necessary when $\ell>1$, but not when
$\ell=1$.
For $\ell=2$, the smoothness conditions are  as stated in
\tref{smooth conn1}, except that $h_1''(L)= 0$ is not required.
Thus the simple expressions for the functions of the positively curved metrics on
$P_1=\Sph^7$ and and the Aloff Wallach space $P_2$ (see \cite{Z2}) cannot be a guide anymore for what a positively curved metric should look like for $\ell>2$.

\bigskip

\begin{center}
 \emph{Curvature of connection metrics.}
 \end{center}

\bigskip

In the remainder of the paper, the metric $\ml \cdot , \cdot \mr $
denotes the $\e$-scaled metric on the total space, as well as the
induced metric on the base.

\smallskip

For the curvature formulas of a connection metric  it turns out to
be useful to introduce the following abbreviations. For the
curvature on the base we set:

\begin{align}\label{curvbase}
L_k:=\ml R_B(Z_k,T)Z_k,T\mr&=-\frac{v_k''}{v_k}\notag \\
M_k:=\ml R_B(Z_i,Z_j)Z_i,Z_j\mr&=\frac{2v_k^2(v_i^2 + v_j^2) -
3v_k^4 + (v_i^2 - v_j^2)^2}{v_i^2v_j^2v_k^2}
-\frac{v_i'}{v_i}\frac{v_j'}{v_j}\\
N_k:=\ml R_B(Z_i,Z_j)Z_k,T\mr&=-2\frac{ v_k'}{v_iv_j} +
\frac{v_i'}{v_i}\frac{v_i^2 + v_k^2 - v_j^2}{v_iv_jv_k} +
\frac{v_j'}{v_j}\frac{v_j^2 +
v_k^2 - v_i^2}{v_iv_jv_k}\notag
\end{align}
where   $(i,j,k)$ is a cyclic permutation of $(1,2,3)$. Notice in
particular that the most basic property a positively curved metric
on the base must satisfy is that the functions $v_i$ have to be
concave. For the principal connection we set:

\begin{align}\label{abbreviations}
 \alpha_i=\frac{h_i'}{2v_i}\quad &,\quad  \beta_{i}
=-\frac{h_i+h_jh_k}{v_jv_k}\notag \\
A_{ij}= (\beta_k,\beta_i,\alpha_i)&\cdot ( 2 \frac{h_j}{v_j},
\frac{v_k^2+v_i^2-v_j^2}{v_iv_jv_k}, \frac{v_j'}{v_j})\\
B_{ij}= (\alpha_k,\alpha_i,\beta_i)&\cdot ( 2
\frac{h_j}{v_j},\frac{v_k^2+v_i^2-v_j^2}{v_iv_jv_k},
\frac{v_j'}{v_j}).\notag
 \end{align}
With this terminology we can now state.

\begin{thm}\label{curv conn}
The curvature tensor of a connection metric, scaled by $\e$ in the
direction of the fibers, is given by
\begin{align*}
R_{X^*_i \bar{Z}_iX_i^* \bar{Z}_i}&=\epsilon^2\alpha_i^2\ ,
&\ \ R_{X_i^*X_j^*X_k^*T}&=0\\
R_{X_i^* \bar{Z}_iX_j^*\bar{Z}_j}&= \,\epsilon \beta_k
-{\epsilon^2} \beta_i\beta_j\  , & R_{X_i^*X_j^*\bar{Z}_kT}&=
-2\epsilon\alpha_k +{\epsilon^2}\left(\alpha_i\beta_j
+\alpha_j \beta_i\right) \\
R_{X_i^* X_j^*X_i^*X_j^*}&=\epsilon\ , &
R_{X_i^*\bar{Z}_jX_k^*T}&=\epsilon\alpha_j
-\epsilon^2 \alpha_i\beta_k \\
R_{X_i^* X_j^*X_i^*\bar{Z}_j}&=0\ ,
&  R_{X_i^*\bar{Z}_j\bar{Z}_kT}&={\epsilon}B_{ij}  \\
R_{X_i^*X_j^*\bar{Z}_i\bar{Z}_j}&=2 \epsilon \beta_k
-\epsilon^2 \left( \beta_i\beta_j\,+\alpha_i\alpha_j \right)\ , &
R_{\bar{Z}_i\bar{Z}_jX_k^*T}
&=-{\epsilon} (B_{ij} +B_{ji}  ) \\
R_{X_i^*\bar{Z}_jX_i^*\bar{Z}_j}&={\epsilon^2} \beta_i^2\ ,
& R_{\bar{Z}_i\bar{Z}_j\bar{Z}_kT}&=N_k
 +\epsilon\cdot \gamma \\
R_{X_i^*\bar{Z}_jX_j^*\bar{Z}_i}&=-\epsilon \beta_k+\epsilon^2
\alpha_i\alpha_j\ ,
 &\ \ R_{X_i^*TX_i^*T}&=\epsilon^2 \alpha_i^2   \\
R_{\bar{Z}_i\bar{Z}_jX_i^*\bar{Z}_j}&=-\epsilon A_{ij} \ ,
&\ \ R_{X_i^*T\bar{Z}_iT}&=- {\epsilon}  \alpha_i'  \\
R_{\bar{Z}_i\bar{Z}_j\bar{Z}_i\bar{Z}_j}&=M_k -3 \epsilon \beta_k^2\ ,
&\ \ R_{\bar{Z}_iT\bar{Z}_iT}&=L_i -3\epsilon \alpha_i^2  \\
\end{align*}
where $i,j,k$  is a cyclic permutation of $(1,2,3)$. All other
components of the curvature tensor are equal to $0$.
\end{thm}

In the above formulas, $\gamma$ is a more complicated expression,
but
it will not enter in the curvature conditions when $\e\to 0$.
\tref{curv conn} follows easily from the curvature tensor for a
general \com\ in \tref{gencurv} in the  Appendix.

\begin{rem}\label{rems}
(a) One easily shows that the curvature $\Omega$ of the principal
connection, and thus the O'Neill tensor $-\frac12\Omega$  of the
Riemannian submersions $P_k\to\Sph^4$ and $Q_k\to\CP^2$, is
determined by: $ \Omega(T,\bar{Z}_i) = \alpha_i\, X_i \ ,\
\Omega(\bar{Z}_i,\bar{Z}_j) =  \beta_k\, X_k $, with $(i,j,k)$
cyclic, and hence $\alpha_i, \ \beta_i$ encode the curvature
$\Omega$. Similarly, from $ (\nabla_A\Omega)(B,C)=2 \sum_i
\langle R( X^*_i ,A)B,C \rangle X_i $ for  any horizontal
$A,B,C$ it follows that $A_{ij}$ and $B_{ij}$ encode the covariant
derivative of $\Omega$. Hence it easily follows that the metric is 3-Sasakian if and only if
$\e=1,\;  \alpha_1=\alpha_2=\beta_3=1,\  \alpha_3=\beta_1=\beta_2=- 1, \; \text{(the signs are determined by the smoothness conditions) } \text{ and }
A_{ij}=B_{ij}=0.$

(b) One easily sees that the bundle is fat, i.e. all vertizontal curvatures are positive, if and only if $
\alpha_i\ne 0, \beta_i \ne 0
$. Notice  that the bundles $P_k$ and
$Q_k$ over $\Sph^4$ all admit fat principal connections since they
carry a 3-Sasakian metric \cite{GWZ}.

(c) If we divide by $\{1\}\times\S^3$, instead of
$\S^3\times\{1\}$, we obtain a second orbifold principal bundle and one easily sees
 that this
bundle cannot have a fat principal connection by using $\beta_i\ne 0$
and smoothness. Similarly, one shows that
 the exceptional manifolds $R$ with slopes $(1,3)$ and $(2,1)$ \cite{GWZ}
 does
 not admit any fat principal connection for both orbifold principal bundles.
The same holds for the cohomogeneity one action on the 7-dimensional
Berger space, where the slopes are $(1,3)$ and $(3,1)$ \cite{Z2}.

(d) All 2-planes which contain the vector $T$ have positive
curvature if and only if
$$
 \qquad\left(\frac{\alpha_i'}{\alpha_i}\right)^2  <
{L_i}\quad,\quad 1\le i \le 3
$$
which follows by looking at all 2-planes of the form
 $(T,X_i^*+s\bar{Z}_i)$ for some $s$. Similarly, a necessary condition for all 2-planes tangent to the
principal orbit to have positive curvature is that
$$
  \qquad  A_{ij}^2<\;
\beta_{i}^2{M_k}\qquad \text{ for   all $i,j,k$ distinct}
$$
which follows by looking at 2-planes of the form $(\bar{Z}_j,\bar{Z}_i+sX_i^*)\ ,
i\ne j$.
\end{rem}

\section{The Curvature Operator  and invariant 4-forms}

\smallskip

In this section we discuss the Thorpe method adapted to
our situation, and our choice
of a suitable auxiliary invariant 4-form to be used to modify the curvature operator.

\smallskip

For the 7-manifolds $P_k,Q_k$ it seems to be quite difficult to
obtain necessary and sufficient conditions for all 2-planes to have
positive curvature in terms of the components of $\hR$. Instead we
develop in the following a set of sufficient conditions which are
easier to verify.

\smallskip

For this, we use a method for estimating sectional curvature due to
Thorpe \cite{Th1}, \cite{Th2} \cite{putmann},  which we now review.
If we denote by $V$ the tangent space at a point in a manifold $M$,
we can regard the curvature tensor as a linear map
$$\hat{R}\colon\Lambda^2 V\to\Lambda^2 V ,$$ which, with respect to
the natural induced inner product on $\Lambda^2 V$, becomes a
symmetric endomorphism. The sectional curvature is then given by:
$$\sec(v,w)=\ml \hat{R}(v\wedge w),v\wedge w\mr
$$
if $v,w$ is an orthonormal basis of the 2-plane they span.

\smallskip

If $\hR$ is positive definite, the sectional curvature is clearly
positive as well. But this condition is exceedingly strong since it
in particular implies that the manifold is covered by a sphere
\cite{BW}. As was first pointed out by Thorpe, one can modify the
curvature operator by using a 4-form $\eta\in \Lambda^4(V)$. It
induces another symmetric endomorphism $\hat{\eta}\colon\Lambda^2
V\to\Lambda^2 V$ via $\ml \hat{\eta}(x \wedge y),z\wedge w\mr
=\eta(x,y,z,w)$. We can then consider the modified curvature
operator $\hR_\eta=\hR+\hat{\eta}$. It satisfies all symmetries of a
curvature tensor, except for the Bianchi identity. Clearly $\hR$ and
$\hR_\eta$ have the same sectional curvature since
$$\ml \hR_\eta(v\wedge w),v\wedge w\mr=
\ml \hR(v\wedge w),v\wedge w\mr + \eta(v,w,v,w)=\sec(v,w)$$

If we can thus find a 4-form $\eta$ with $\hR_\eta>0$, the sectional
curvature is positive.  Thorpe showed \cite{Th2} that in dimension
4,  one can always find a 4-form such that the smallest eigenvalue
of $\hR_\eta$ is also the minimum of the sectional curvature, and
similarly a possibly different 4-form such that the largest
eigenvalue of $\hR_\eta$ is the maximum of the sectional curvature.
This is not the case anymore in dimension bigger than 4 \cite{Zo}.
Nevertheless this can be an efficient method to estimate the
sectional curvature of a metric. In fact, P\"uttmann \cite{putmann}
used this to compute the maximum and minimum of the sectional
curvature of all positively curved homogeneous spaces, which are not
spheres. It is peculiar to note though that this method does not
work to determine which homogeneous metrics on $\Sph^7$ have
positive curvature, see \cite{VZ1}.

\bigskip

To illustrate this method, we first derive necessary and sufficient
conditions for positive curvature on the base, although in the end,
positive curvature on the base will  be a consequence of the
positivity of the determinants in Section 5.

Using the orthonormal basis  $Z_i, T$ of the tangent space along the
normal geodesic described in Section 2, and letting $d\theta_i$ be
the one forms dual to $Z_i$, we have:

\smallskip

\begin{thm}\label{base pos}
The cohomogeneity one metric
$$ds^2=dt^2+v_1^2(t)d\theta_1^2+v_2^2(t)d\theta_2^2+v_3^2(t)d\theta_3^2$$
has positive curvature if and only if
$$
L_i>0 \quad , \quad M_i>0 \quad \text{and} \quad \left|N_i-N_j\right|<\sqrt{L_iM_i}+
\sqrt{L_jM_j}.
$$
where $L_i,M_i,N_i$ are the curvature components defined in
\eqref{curvbase}.
\end{thm}
\begin{proof}
Using the orthonormal basis $Z_1,Z_2,Z_3,T$ of the tangent space
$V$, we write $\Lambda^2V$ as the direct sum of the following three
2-dimensional subspaces:
$$\left\{Z_1\wedge Z_2\ , \ Z_3\wedge T\right\}\ ,\
 \left\{Z_2\wedge Z_3\ , \ Z_1\wedge T\right\}\ ,\
 \left\{Z_3\wedge Z_1\ , \ Z_2\wedge T\right\}.$$
 Notice that these are in fact inequivalent to each other under the
 action of the isotropy group $\{\pm 1, \pm i,\pm j,\pm k\}$ and hence the
 curvature operator $\hat{R}\colon\Lambda^2 V\to\Lambda^2 V$
 breaks up into three $2\times 2$ blocks. If we modify this
 curvature operator with the 4-form $\eta=d\cdot Z_1\wedge Z_2\wedge Z_3\wedge
 T$, the modified operator $\hR_\eta=\hR+\hat{\eta}$
  consists of the following
 blocks
 \begin{equation}\label{2x2}
\begin{pmatrix}
M_i &N_i+d\\
N_i+d &L_i
\end{pmatrix} \qquad i=1,2,3.
 \end{equation}
Assuming that $L_i>0,\ M_i>0$ this matrix is positive definite if
and only if $d$ lies in the interval $I_k:=[C_k-R_k,C_k+R_k]$ with
center $C_k=-N_k$ and radius $R_k=\sqrt{L_kM_k}$. For $\hR_\eta$ to
be positive definite, we thus need to find a $d$ that lies in the
intersection of these three intervals. On the other hand, the
intervals $I_k$ intersect if and only if $|C_i-C_j|<R_i+R_j$ for all
$i<j$. Since, as was shown by Thorpe, this method in dimension $4$
always finds the minimum of the sectional curvature for suitable
$d$, the result follows.
\end{proof}

\smallskip

For the 7-manifolds $P_k$, we use the fact that the curvature
operator $\hat{R}$ commutes with any isometry and hence the action
of the isotropy group $H$. We therefore choose the basis of
$\Lambda^2 V $, where $V=\spam\{T,X_i^*,\bar{Z}_i\}$, as follows:
$$
\{X_1^*\wedge \bar{Z}_1,\ X_2^*\wedge \bar{Z}_2,\ X_3^*\wedge \bar{Z}_3  \}
$$
$$
\{X_1^*\wedge X_2^*,\ X_1^*\wedge \bar{Z}_2,\ \bar{Z}_1 \wedge X_2^* ,
\ X_3^*\wedge T,\  \bar{Z}_1\wedge \bar{Z}_2 ,\ \bar{Z}_3\wedge T
\}
$$
$$
\{X_2^*\wedge X_3^*,\ X_2^*\wedge \bar{Z}_3,\ \bar{Z}_2\wedge X_3^*,
\ X_1^*\wedge T, \ \bar{Z}_2\wedge \bar{Z}_3 , \ \bar{Z}_1\wedge T  \}
$$
$$
\{X_3^*\wedge X_1^*,\ X_3^*\wedge \bar{Z}_1,\ \bar{Z}_3\wedge X_1^*,
 \ X_2^*\wedge T, \ \bar{Z}_3\wedge \bar{Z}_1 , \ \bar{Z}_2\wedge T   \}
$$
The action of $H$ is trivial on the first space, and the action on
the remaining 3 spaces are inequivalent to each other, whereas on
each individual space, it acts the same on all six vectors. Thus the
curvature operator can be represented by a matrix that splits up
into one $3\times 3$ block, which we denote by $A_0$, and three
$6\times 6$ blocks, denoted $A_{12}, \ A_{23}$ and $A_{31}$
respectively.

\smallskip

The needed considerations for the $6\times 6$ blocks can easily be
reduced further to the lower $5\times 5$ blocks
 by using the  following observation. If one uses
a Cheeger deformation by an isometric action of $G=\SU(2)$ or
$\SO(3)$ on a Riemannian manifold, then as long as all  2-planes
whose projection onto the $G$ orbits is one dimensional are
positively curved, the Cheeger deformation will automatically
produce positive sectional curvature on all 2-planes, when the
metric is shrunk sufficiently in the orbit direction (see e.g.
\cite{Mu}, \cite{PW2}).
 If
one applies this observation to the $\S^3$ action on the base, it
shows that all curvatures will eventually become positice as long as
 $\sec(T,Z_i)$ is positive, i.e. $v_i$  is concave \eqref{curvbase}. In particular, there are no
obstructions to obtaining positive curvature on the base for any
$\ell$. When applied to a deformation of the metric on the
7-manifold by the first factor in $\S^3\times\S^3$, it shows that
only the lower 5x5 block is needed.
  In the following $A_{ij}$ will denote this lower
$5\times 5$ block. Notice though that such a Cheeger deformation
stays within the class of connection metrics, in fact corresponds
precisely to letting $\e\to 0$. We also point out that our proof
in Section 5 works just as easily for the 6x6 matrix directly as well.

\smallskip

We now  modify $\hR$ with a 4-form $\eta$ on $V$.  As was observed
by P\"uttmann, the 4-form $\eta$ can be assumed to be invariant
under the isometry group and hence we choose $\eta$ to be invariant
under the action of $H=\triangle Q$ on $V$. One easily sees that
such 4-forms are of the form
\begin{align}\label{4form}\eta=&a_3 X_1^*\wedge X_2^*\wedge
\bar{Z}_1\wedge \bar{Z}_2 +a_1 X_2^*\wedge X_3^*\wedge \bar{Z}_2\wedge \bar{Z}_3 +a_2 X_3^*\wedge
X_1^*\wedge \bar{Z}_3\wedge Z_1 \notag \\
&+b_2 X_1^*\wedge \bar{Z}_2\wedge X_3^*\wedge T+
b_1 \bar{Z}_1\wedge X_2^*\wedge X_3^*\wedge T   + b_3 X_1^*\wedge X_2^*\wedge \bar{Z}_3\wedge T \\
&+ c_1 X_1^*\wedge \bar{Z}_2\wedge \bar{Z}_3\wedge T  +
c_2 \bar{Z}_1\wedge X_2^*\wedge \bar{Z}_3\wedge T + c_3 \bar{Z}_1\wedge \bar{Z}_2\wedge X_3^*
\wedge T \notag \\
& +
d_1 X_1^*\wedge X_2^*\wedge X_3^*\wedge T  +
d_2 \bar{Z}_1\wedge \bar{Z}_2\wedge \bar{Z}_3  \wedge T\notag
\end{align}
for some constants $a_i,b_i,c_i,d_i$, which we will call
P\"{u}ttmann parameters from now on.

\smallskip

The optimal choice of these P\"uttmann parameters is in general a
difficult problem. For our metrics we set
\begin{align}\label{parameters} a_i&= \epsilon\,
\beta_i-\epsilon^2\beta_j\beta_k \notag \\
b_i&=-\e
\alpha_i+\frac12\e^2(\alpha_j\beta_k+\alpha_k\beta_j)\\
c_i&=0, \quad d_1=0,\quad d_2=-N_2\notag.
\end{align}

We shortly motivate this choice. Using the curvature formulas in
\tref{curv conn}, we see that the matrix $A_0$ takes on the form
$$A_0=\begin{pmatrix} \epsilon^2 \alpha_1^2 &\epsilon\, \beta_3-
\epsilon^2\beta_1\beta_2-a_3 & \epsilon\,\beta_2-
\epsilon^2\beta_1\beta_3-a_2\\
 \epsilon\,\beta_3-
\epsilon^2\beta_1\beta_2-a_3 &\epsilon^2 \alpha_2^2 &\epsilon\, \beta_1-
\epsilon^2\beta_2\beta_3-a_1\\
 \epsilon\,\beta_2-
\epsilon^2\beta_1\beta_3-a_2 &\epsilon\, \beta_1-
\epsilon^2\beta_2\beta_3-a_1 &\epsilon^2 \alpha_3^2
\end{pmatrix}
$$
 Our choice of $a_i$ makes this matrix diagonal, and  $\alpha_i\ne 0$
 then implies that it is positive definite.  Each one of the parameters $b_i$ and $c_i$ occur in  one $2\times 2$ minor (centered along the diagonal) of each $A_{ij}$ matrix. As in the proof of \tref{base pos}, they have positive determinant if and only if three intervals intersect, which
 suggests a reasonable choice for their values.
The P\"uttmann parameter $d_1$ only corresponds to
entries in the curvature matrix that are $0$ for a connection
metric. We thus set $d_1=0$.
 The last P\"{u}ttmann parameter $d_2$
   is
contained in the three lower $2\times 2$ blocks
$$
\begin{pmatrix}
M_i-3\e \beta_i^2 & N_i+d_2+\e \gamma\\
N_i+d_2+\e \gamma & L_i-3\e\alpha_i^2
\end{pmatrix}
$$
whose positivity is guaranteed when the  modified  curvature
operator on the base is positive definite, as $\e\to 0$. For our
metrics, it turns out that $d=-N_2$ is sufficient.

\smallskip

One now easily shoes that the lower $5\times 5$ block of the thus
modified curvature matrix $A_{ij}$ takes on the form
$$\hspace{-15pt}\begin{pmatrix}
  \e^2 \beta_i^2 &\epsilon^2(\beta_i\beta_j-\alpha_i\alpha_j)
  &
   \frac12 \e^2(\alpha_k\beta_i-\alpha_i\beta_k)&{-\e}\, A_{ij}
    & \e B_{ij}\\
\epsilon^2(\beta_i\beta_j-\alpha_i\alpha_j)
& \e^2 \beta_j^2 &\frac12 \e^2(\alpha_k\beta_j-\alpha_j\beta_k)&{-\e}\,
A_{ji} & \e B_{ji}\\
 \frac12 \e^2(\alpha_k\beta_i-\alpha_i\beta_k)&\frac12
 \e^2(\alpha_k\beta_j
 -\alpha_j\beta_k)
 &\e^2\alpha_k^2&-\e(B_{ij}+B_{ji}) &-\e\alpha_k'\\
 {-\e}\, A_{ij}&{-\e}\, A_{ji}
 & -\e(B_{ij}+B_{ji})&M_k-3\e \beta_k^2 &N_k+\e \gamma +d_2\\
  \e B_{ij}& \e B_{ji}
 &-\e\alpha_k'&N_k+\e \gamma +d_2& L_k-3\e\alpha_k^2
\end{pmatrix}$$
For example, the second entry in the first row is equal to
$\langle\hat{R}(X_1^*\wedge \bar{Z}_2), \bar{Z}_1 \wedge X_2^* \rangle +\eta(
X_1^*, \bar{Z}_2 , \bar{Z}_1 , X_2^* )=
\epsilon \beta_3
-{\epsilon^2} \alpha_1\alpha_2- a_3={\epsilon^2} (\beta_1\beta_2-\alpha_1\alpha_2 )
$ and similarly for the other entries.

\smallskip

In the above matrix we can remove an $\epsilon$ from the first 3 rows and columns and, as $\epsilon\to 0$, replace the lower $2\times 2$ block by the one in \eqref{2x2}.
We need to show that this new matrix, which does not depend on $\e$ anymore, is positive definite. By Sylvester's theorem it
suffices to show  that the determinants of
the $k\times k$ minors in the upper block (consisting of rows and
columns 1 through $k$) are positive for $k=1,\dots, 5$. Since this also implies that $\alpha_k\ne 0$, $A_0$ positive definite as well.

\smallskip

It is also instructive to notice that under the assumption that the
metric is 3-Sasakian,  all but one of the off diagonal components of
the modified curvature matrix $A_{ij}$ vanish, due to the above
choice of the P\"{u}ttmann parameters. Hence the modified curvature
operator is positive definite as long as the sectional curvature  on
the base is positive, thus recovering the main theorem in \cite{De1}
in our context. It is thus useful to stay close to a 3-Sasakian metric.

\section{Metric on the base and principal connection}

For our connection metric, the functions $v_i$ and $h_i$ are given by piecewise polynomials, which we choose as follows:

\begin{center}
\bf $\mathbf v_i$ functions.
\end{center}

\bigskip

\begin{align}\label{vfunctions}
v_1&=\left\{
  \begin{array}{ll}
    4\,t - 10\,t^{3} , & \hbox{$\qquad \ \ \ \ \ \  \, 0<t<1/10$} \\
    p_1(t), & \hbox{$\qquad \ \  1/10<t<1/2$} \\
  \dfrac {5}{4}  - 3\,(t - L)^{
2} + (t - L)^{3}, & \hbox{$\qquad \ \  \ 1/2<t<L$}
  \end{array}
\right. \notag \\
v_2&=\left\{
  \begin{array}{ll}
 149/200 -
 \dfrac {11}{9} \,t -\dfrac {1}{10} \,t^{2}  -
 \dfrac {1}{25} \,t^{3}, & \hbox{$\ \ \ \ \, 0<t<1/10$} \\
   p_2(t), & \hbox{$1/10<t<1/2$} \\
      - \dfrac43 \,(t - L)+\dfrac{3}{10} \,(t - L)^{
3} , & \hbox{\ \ $1/2<t<L$}
  \end{array}
\right.\\
v_3&=\left\{
  \begin{array}{ll}
    149/200 +
 \dfrac {11}{9} \,t -\dfrac {1}{10} \,t^{2}  -
  \dfrac {7}{10} \,t^{3}, & \hbox{$\ \ \ \  \, 0<t<1/10$} \\
    p_3(t), & \hbox{$1/10<t<1/2$} \\
   \dfrac {5}{4}  - 3\,(t - L)^{
2} -3 (t - L)^{3}, & \hbox{  \ $1/2<t<L$}
  \end{array}
\right.\notag
\end{align}
where $L=\frac{58}{100}$.

The polynomials $p_i(t)$ are chosen to be the unique degree 5
polynomials such that the new piecewise function is $C^2$ at $t=1/10$
and $t=1/2$. From the smoothness conditions in \tref{smooth conn1}, one
sees that the metric is  $C^2$ at $t=0$ and $t=L$. The third
derivatives  though show that the metric is not $C^3$.
For the principal connection we choose:

\begin{center}
\bf $\mathbf h_i$ functions.
\end{center}

\begin{align}\label{hfunctions}
h_1&=\left\{
  \begin{array}{ll}
    -1+4t^2-4t^4, & \hbox{$\quad \ \ \ \ \ \, 0<t<1/10$} \\
    q_1(t), & \hbox{\quad \ $1/10<t<1/2$} \\
    \dfrac{31}{12}(t-L)-\dfrac{16}{7}(t-L)^3, & \hbox{\ \ \quad $1/2<t<L$}
  \end{array}
\right.\notag \\
h_2&=\left\{
  \begin{array}{ll}
    \dfrac{21}{17}+\dfrac{16}{11}t-\dfrac{21}{17}t^2+\dfrac{1}{10}t^3,
& \hbox{$\ \ \ \ \, 0<t<1/10$} \\
    q_2(t), & \hbox{$1/10<t<1/2$} \\
    \dfrac53-\dfrac{4}{3}(t-L)^2+\dfrac14(t-L)^4, & \hbox{\ $ 1/2<t<L$}
  \end{array}
\right.\\
h_3&=\left\{
  \begin{array}{ll}
    \dfrac{21}{17}-\dfrac{16}{11}\, t-\dfrac{21}{17}t^2-\dfrac{1}{10}t^3,
 & \hbox{$ \ \ \ \ \, 0<t<1/10$} \\  \vspace{5pt}
    q_3(t), & \hbox{$1/10<t<1/2$} \\
    -\dfrac{31}{12}(t-L)+\dfrac{20}{11}(t-L)^3, & \hbox{\ $1/2<t<L$}
  \end{array}
\right.\notag
\end{align}
where
$q_i(t)$ are again the unique degree 5 polynomials such that the principal connection is $C^2$. It is interesting to note that if we
choose as a principal connection the Levi-Civita connection of the
metric in \eqref{vfunctions}, the resulting metric on $P_2$ does not
have positive curvature.

\smallskip

This metric was obtained as follows. We first find piecewise polynomial functions $h_i$ such that the principal connection is a very close approximation of the Levi-Civita principal connection associated to the  Hitchin metric. For the Hitchin metric on the base, the functions $v_2$ and $v_3$ are not concave at $t=0$ as required by positive curvature. We fist stay close to the convex hull of the Hitchin metric and then deform it further in order to satisfy the necessary and sufficient conditions in \tref{base pos} in order to produce a metric on the base with positive curvature. One then makes further changes to this metric, keeping the principal connection the same, until the necessary conditions in \rref{rems} (d) are satisfied. No further changes were necessary in order to make the determinants described in Section 3 positive.
In Figure 1 we give a picture of the $v_i$ functions together with the Hitchin functions on the left, and the $h_i$ functions on the right.

\begin{figure}[!ht]
\begin{center}
\includegraphics[width=3in,height=3in,angle=-90]{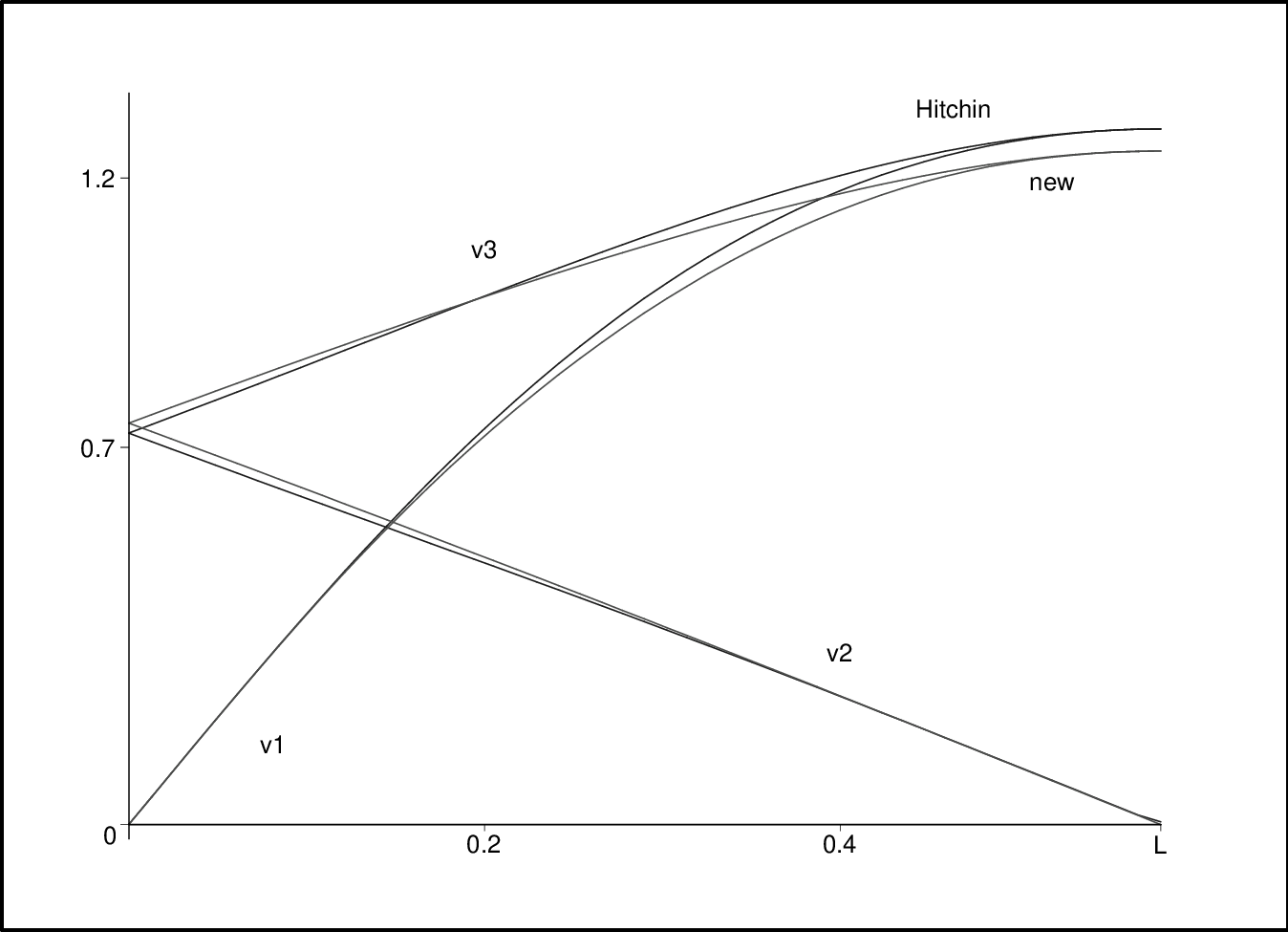}\qquad
\includegraphics[width=3in,height=3in,angle=-90]{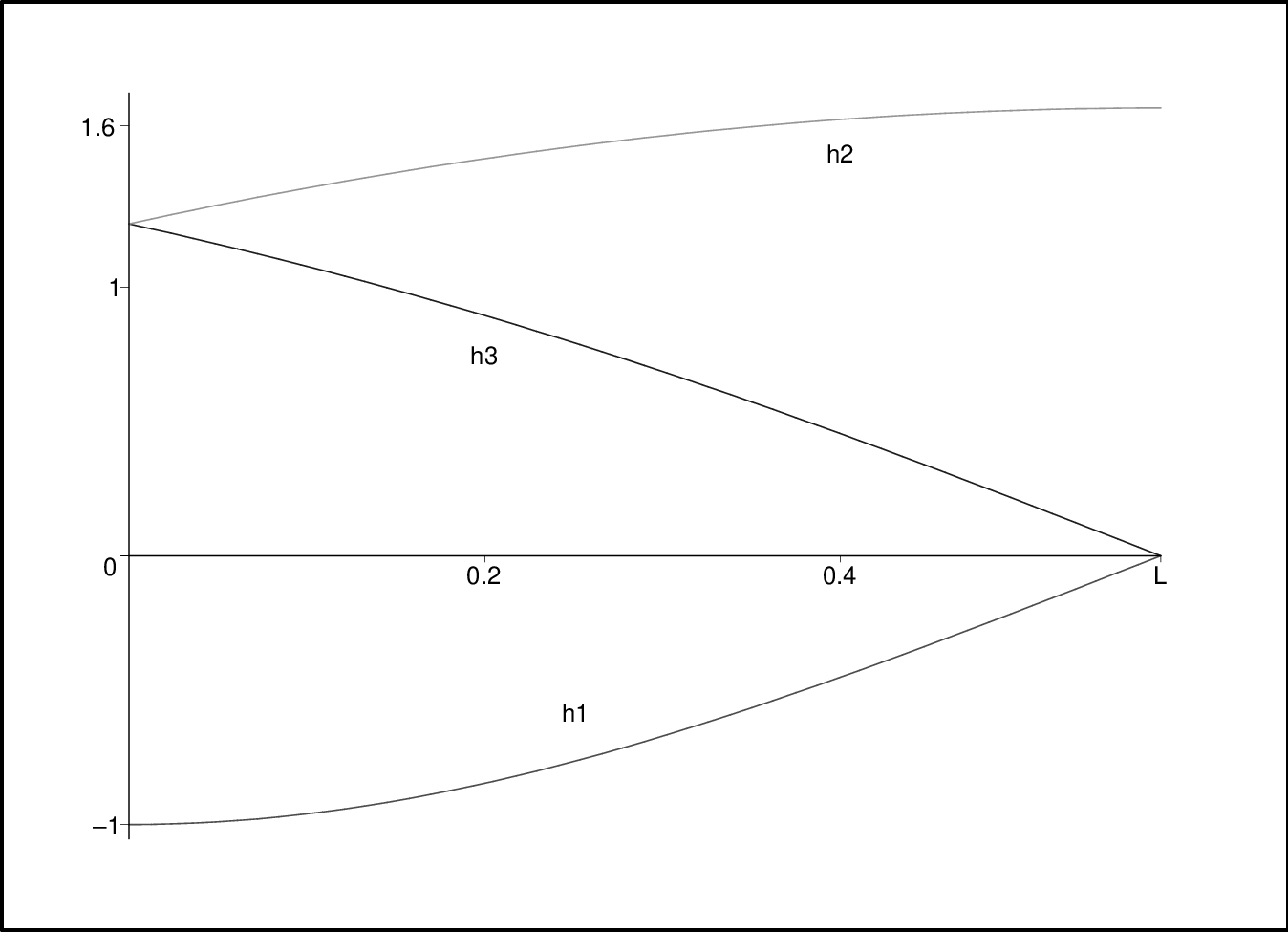}
\end{center}
\caption{ $v_i$ functions  and Hitchin functions, as well as $h_i$.}
\end{figure}

\section{Positivity of the determinants}

As explained in Section 3, our proof will show that the modified
curvature operator is positive definite by choosing the 4-form as in
\eqref{4form}, with P\"{u}ttmann parameters  \eqref{parameters}.
For this we need to prove that the determinants of the $k\times k$ minors in
the upper block of  the $5\times 5$ matrices
$A_{ij}$ (consisting of rows and columns 1 through $k$) are
positive for $k=1,\dots, 5$.  We divide the interval $[0,L]$ into
the three subintervals
$[0,\frac{1}{10}],\ [\frac{1}{10},\frac{1}{2}]$ and $[\frac{1}{2},\frac{58}{100}]$ on
which our metric is defined by polynomials. Each determinant is thus
a rational function in the arclength parameter $t$ whose coefficients are rational as well. To show that it is
positive, we use a theorem due to Sturm (see \cite{Ja}) that gives a
simple procedure for counting zeroes of a polynomial with rational
coefficients on a closed interval in terms of a Euclidean algorithm.

To be specific, let $p(t)$ be a polynomial with integer
coefficients.
One inductively defines a finite sequence of polynomials (Sturm's sequence) with $p_1=p(t)$, $p_2=p'(t)$ and $p_{i+1}=-rem(p_i,p_{i-1})$ where $rem()$ is the remainder of the polynomial division. If $p(t)$ and $p'(t)$ have no common zeros, the last remainder $p_k(t)$ is a nonzero constant. Otherwise $p_k(t)=0$ and $p_{k-1}(t)$ is a common factor of $p(t)$ and $p'(t)$, corresponding to double roots of $p(t)$, and
thus $p$ and  $p/p_{k-1}$ have the same zeroes. In this case the Sturm sequence for $p$ is that of $p/p_{k-1}$.
 Now Sturm's theorem states that if $p_i(t),\
i=1\dots k$ is the Sturm sequence of $p(t)$, then the number of real
zeroes in the half open interval $(a,b]$ is equal to the difference
in the number of sign changes (not counting any zeroes) in the
sequence $[p_1(a),\dots , p_k(a)]$ and the sequence $[p_1(b),\dots ,
p_k(b)]$. Since the endpoints of the 3 intervals are rational numbers, the same is true for the sequences $[p_1(a),\ldots,p_k(a)]$ and
$[p_1(b),\ldots, p_k(b)]$. Thus the proof only deals with calculations involving rational numbers.

\smallskip

The degrees of the determinant polynomials are quite large, but one can easily modify the above procedure to significantly reduce these degrees, so that the proof  can  be carried out by hand:  In each subinterval we translate the parameter $t$ in $v_i$ and $h_i$ so that
the determinants are polynomials $f(s)$ of a variable $s$ defined in $[0,S]$. We define a new polynomial
$g(s)$ that collects all the monomials in $f(s)$ with negative coefficient. Then the derivatives $g^{(k)}(s)$ are non-positive
decreasing functions of $s$ for any $k\geq 0$. In particular $g^{(k)}(s)\geq g^{(k)}(S)$. If $f_k(s)$ denotes the truncated  polynomial collecting the monomials of $f(s)$ of order $\leq k-1$, the remainder at a point $s\in [0,S]$ is given by Taylor's formula and can be estimated by
$$\frac{1}{k!} f^{(k)}(s_0(s)) s^k\geq \frac{1}{k!} g^{(k)}(s_0(s)) s^k\geq \frac{1}{k!} g^{(k)}(S) S^k=R$$
independently of $s$. Since we choose $S$ to be rational, $R$ is  rational  as well. Then
$$f(s)\geq f_k(s)+R$$
and we prove that this last polynomial is positive in $[0,S]$ using Sturm's theorem as above. In this procedure we need to divide the interval in the middle into 4 further subintervals. It turns out that  the degree of
$f_k$ will be smaller than $20$, in most cases smaller than $10$.
   A Maple program that
carries out these calculations is made available at
www.math.upenn.edu/\~{}wziller/research.html. Notice that Maple can do this  symbolically, i.e., no floating point operations are used.

\smallskip

For illustrative purposes, we  draw the graph of the
determinants in Figure 3. The determinants of the $4\times 4$ and
$5\times 5$ minor in the matrix $A_{23}$ is not included in the
second picture since its values lie between 5 and 25.

\begin{figure}[!ht]
\begin{center}
\includegraphics[width=3.5in,height=3in,angle=-90]{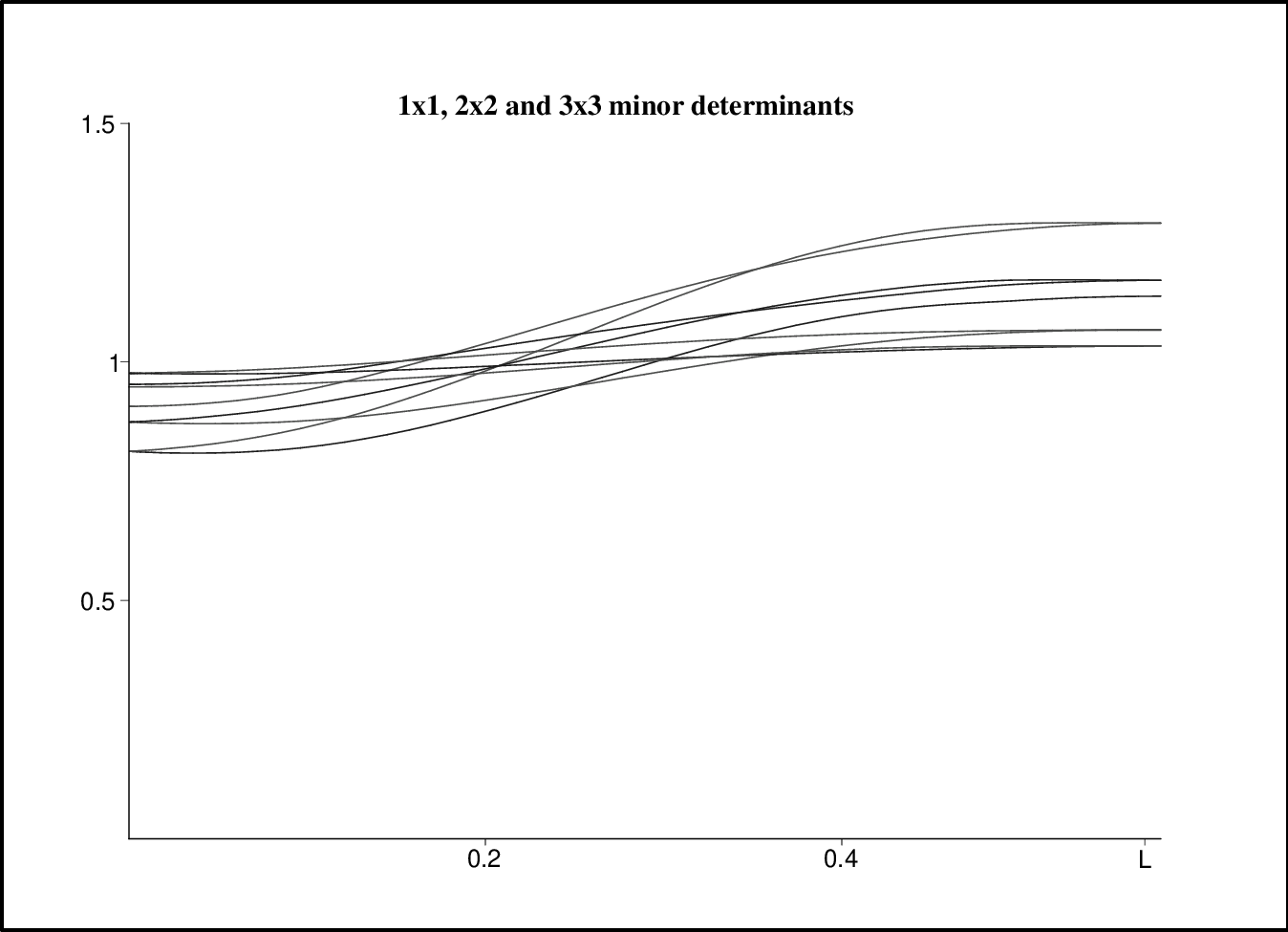}\quad
\includegraphics[width=3.5in,height=3in,angle=-90]{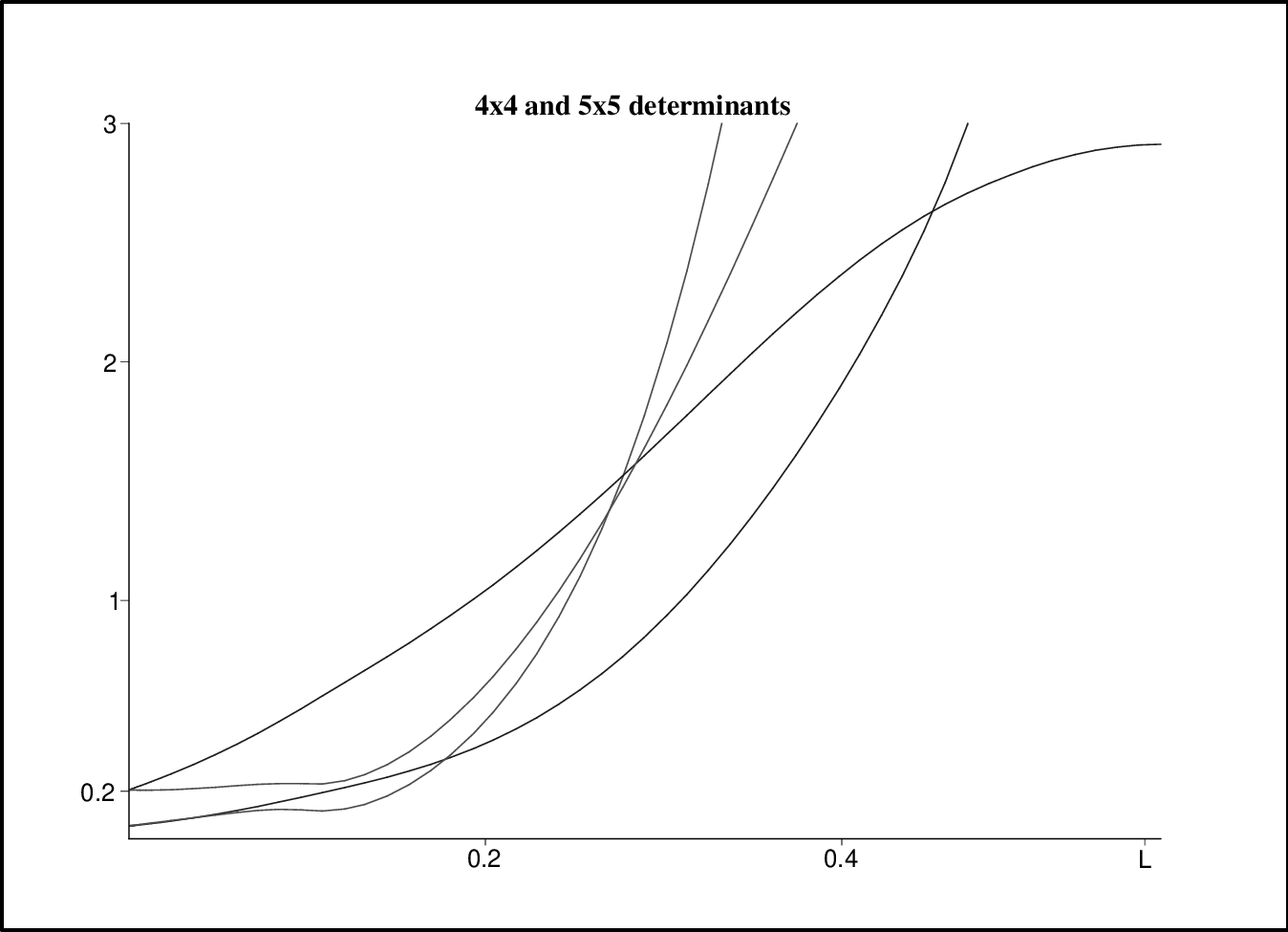}\quad
\end{center}
\caption{Determinants of all 5 minors in $A_{ij}$ }
\end{figure}

\vspace{100pt}

\section{Appendix }

\bigskip

\begin{center}
 \emph{ Smoothness of metrics on $P_k$ and $Q_k$.}
 \end{center}

\bigskip

At the endpoints $t=0$ and $t=L$, the principal orbits collapse and
hence the functions need to satisfy certain smoothness conditions.
Smoothness conditions for cohomogeneity one manifolds have been discussed, e.g., in \cite{BH} and \cite{EW}. For convenience of the reader, we present here an elementary proof in the case of codimension 2 orbits.

 We do this first for arbitrary slopes $(p,q)$ of the circle $\Ko\subset\S^3\times\S^3$  since
this makes the discussion more transparent and for convenience we
assume the singular orbit occurs at $t=0$. We will also assume that
only the inner products in \eqref{functions} are non-zero, although
in this generality this is not necessarily true for all $G$
invariant metrics.

\begin{thm}\label{gen smooth}
Let $H\subset K= \{e^{ip\theta},e^{iq\theta}\}\cdot H\subset
G=\S^3\times\S^3$ be a singular orbit at $t=0$ with $H$ finite,
$|H\cap \Ko|=k$, and $\gcd(p,q)=1$. Assuming that $f_i,g_i,h_i$ are
the only non-vanishing inner products, the metric is smooth if and
only if:
\begin{itemize}\label{smoothsingorbit}
\item[(a)] For the collapsing functions $f_1,g_1,h_1$ we have:
\begin{align*}
f_1,g_1,h_1 \text{ are even at } t&=0 \text{ and }\\
p\; f_1=-q\; h_1\quad , \quad  q\; g_1&=-p\; h_1\\
p^2\,f_1''+q^2\,g_1''+2pq\, h_1''&=
2k^2\\
\end{align*}
\item[(b)] For the remaining functions we have:
\begin{align*}
f_2+f_3&=\phi_3(t^2)
&f_2-f_3&=t^{ \frac{4|p|}{k}}\phi_4(t^2), \\
g_2+g_3&=\phi_5(t^2)
&g_2-g_3&=t^{ \frac{4|q|}{k}}\phi_6(t^2),\\
h_2+h_3 &=t^{ \frac{2|q-p|}{k}}\phi_7(t^2)
&h_2-h_3&=t^{ \frac{2|q+p|}{k}}\phi_8(t^2).
\end{align*}
'\end{itemize} where $\phi_i$ are smooth functions. When the
exponent in $t$ is a fraction, the right hand side should be set
to $0$.
\end{thm}

\begin{proof}
First notice that by $G$ invariance of the metric, it is smooth as
long as the restriction to a slice $V$, i.e. a disc orthogonal to
the singular orbit $G/K$, is smooth. The metric is defined along a
line in $V$ and needs to be extended by $\Ko$ invariance. Thus the
issue is wether this extension is smooth at $0\in V$.

In the following sequence of lemmas we do not yet make any
assumption on the group $G$, but only assume that the singular orbit
$G/K$ has codimension 2. We  start with the metric on the slice $V\simeq\R^2\simeq\C^2$.  If
$\Ko=\{e^{i\theta}\mid 0\le \theta\le
 2\pi\}$, the action on the slice is given by multiplication with $e^{ik\theta}$ since $|H\cap K^0|=k$.
 If we let $Z=\frac{d}{d\theta}\in\fk$ and $f(t)=|Z^*(c(t))|^2$, the usual proof for the smoothness of a metric in polar coordinates shows that

\begin{lem}\label{slice} With the notation defined above, the
metric on $V$ is smooth if and only if $f(t)=k^2 t^2 +t^4 \phi(t^2)$
for some smooth function $\phi$.
\end{lem}

Let $\fg=\fm\oplus\fk$ and $\fk=\fh\oplus\fp$ be $Q$ orthogonal
decompositions. Notice that $\dim\fp=1$ since $K/H$ is one
dimensional. Furthermore, since $\Ko=\S^1$, the $\Ko$ irreducible
modules in $\fm$ are either trivial or two dimensional. In the case
of a trivial module we have:
\begin{lem}\label{trivial} Let $X$ be a vector in a one dimensional
subspace $\fm_0$ of $\fm$ on which $\Ko$ acts trivially, and
$Z\in\fp$ as above. If the representation of  $H$ on $\fm_0$ and
$\fp$ are equivalent, then the metric is smooth if and only if, in
addition to $f(t)=|Z^*|_{c(t)}^2=k^2 t^2 +t^4 \phi_1(t^2)$, we have
$$|X^*|_{c(t)}^2=\phi_2(t^2),\; \text{ and } \ml Z^*,
X^*\mr_{c(t)} = t^2\phi_3(t^2)$$ for some smooth functions $\phi_i$.
If the representation of $H$ on $\fm_0$ and $\fp$ are inequivalent,
we have  $\ml Z^*, X^*\mr_{c(t)}=0$.
\end{lem}

Indeed, since $\Ko$ fixes $X$, the inner products are   even
functions of $t$. Furthermore, $Z^*$ and $X^*$ are
orthogonal at $c(0)$ since the slice $V$ is orthogonal to the
singular orbit $G/K$. The proof now is similar to the proof of \lref{slice}.

\smallskip

 If $\fm_1\subset\fm$ is a 2-dimensional $\Ko$ irreducible module,
 invariant under $H$,
  we identify $\fm_1\simeq \C$, in which case the
action of $\Ko$ will be given by $z\to e^{i d \theta} z$ for some
$d\in \Z$. By possibly changing the
order of the basis in $\fm_1$ and $\fp$ if necessary, we may assume
that $d>0, k>0$.  In the natural basis of $\C$ we let
$g_{11},g_{22},g_{12}$ be the inner products of the basis vectors
along $c(t)$.

\begin{lem}\label{onemod} With the notation defined above, the
restriction of the metric to the $K$ irreducible module $\fm_1$
admits a smooth extension to the singular orbit if and only if $k$
divides $2d$ and
$$(g_{11}-g_{22})(t)=t^{2 \frac dk}
\phi_1(t^2),\qquad g_{12}(t)=t^{2 \frac dk} \phi_2(t^2),
 \qquad (g_{11}+g_{22})(t)=\phi_3(t^2)$$
where $\phi_i(t)$ are smooth functions. When the
exponent in $t$ is a fraction, the right hand side should be set
to $0$.
\end{lem}
\begin{proof} We extend the inner products $g_{ij}$ to functions
along the slice $V$.
 Let $P$ be the restriction
of the metric tensor to $\fm_1$ and $R(\theta)$ represent a rotation
by $\theta$. Since $e^{i\theta}\in \Ko$ acts by $e^{ik \theta}$ on
the slice $V\simeq\C$ and by  $e^{id \theta}$ on $\fm_1$,  the $\Ko$
invariance of the metric can be written in matrix form:
$$
P(e^{ ik\theta} p)= R(d\theta)P(p)R(-d\theta),\; p\in V
$$
In other words,
$$\left\{\begin{array}{l}
g_{11}(e^{ik \theta} p)=\cos^2(d \theta)\,g_{11}(p)+\sin^2(d \theta)\,g_{22}(p)-2\sin(d\theta)\cos(d\theta)\,g_{12}(p)\\
g_{12}(e^{ik \theta} p)=(\cos^2(d \theta)-\sin^2(d \theta))\,g_{12}(p)-\sin(d\theta)\cos(d\theta)\,(g_{22}(p)-g_{11}(p))\\
g_{22}(e^{ik \theta} p)=\sin^2(d \theta)\,g_{11}(p)+\cos^2(d \theta)\,g_{22}(p)+2\sin(d\theta)\cos(d\theta)\,g_{12}(p)\\
\end{array}\right.$$
If we set $w=(g_{11}-g_{22})+2 i g_{12}$, then one easily shows that
the above equations are equivalent to:
$$\left\{\begin{array}{l}
(g_{11}+g_{22})(e^{k i\theta} p)=(g_{11}+g_{22})(p))\\
w(e^{k i\theta} p)=e^{2 d i \theta} w(p)
\end{array}\right.$$
Notice that if $w$ vanishes identically, this implies that the
metric is smooth if and only if $g_{11}=g_{22}$ is even. If not, the
second equation shows that $w$ is only well defined when $k$ divides
$2d$. It then reduces to
$$w(t e^{i\theta})=e^{2 \frac dk i \theta} w(t),\ t\in\R$$
or equivalently
$$z^{-2 \frac dk} w(z)=t^{-2 \frac dk} w(t).$$
If $g(z)=z^{-2 \frac dk} w(z)$, it follows  that $g(z)$ and
$g_{11}+g_{22}$ are $K$-invariant functions on $V$. Such
functions  admit a smooth extension to the origin if and only if they are even and thus
$$w(z)=z^{2 \frac dk}( \phi_1({t}^2) + i  \phi_2({t}^2)),\qquad (g_{11}+g_{22})(z)=\phi_3({t}^2)$$
Separating the real and the imaginary part and restricting to the
normal geodesic proves our claim.
\end{proof}

\bigskip

Next we deal with the case of two  irreducible modules
$\fm_1\simeq\C$ and $\fm_2\simeq\C$ under the action of $\Ko$, whose
restriction to $H$ are equivalent. Inner products between vectors in
$\fm_1$ and $\fm_2$ are thus not necessarily 0. We choose bases
$\{v_1,w_1\}$ of $\fm_1$ and $\{v_2,w_2\}$ of $\fm_2$ such that the
action of $\Ko$ on $\fm_1$ is given by $z\to e^{i d \theta} z$ and
on $\fm_2$ by $z\to e^{i d' \theta} z$ for some $d,d'\in \Z$, and we
can  again assume that $d,d',k$ are all positive. The inner products
between $\fm_1$ and $\fm_2$ are then determined by
$h_{11}=g(v_1,v_2)$, $h_{12}=g(v_1,w_2)$, $h_{21}=g(w_1,v_2)$ and
$h_{22}=g(w_1,w_2)$ along $c(t)$. As in the proof of \lref{onemod}, one easily shows:

\begin{lem}\label{twomod} With the notation as above, the scalar products
between elements of  $\fm_1$ and $\fm_2$ admit a smooth extension to
the singular orbit if and only if $k$ divides $d\pm d'$ and
$$(h_{11}+h_{22})(t)=t^{\frac {d'-d}k} \phi_1(t^2),
\qquad (h_{11}-h_{22})(t)=t^{\frac {d'+d}k} \phi_2(t^2)$$
$$(h_{12}-h_{21})(t)=t^{\frac {d'-d}k} \phi_3(t^2),
\qquad (h_{12}+h_{21})(t)=t^{\frac {d'+d}k} \phi_4(t^2)$$
where $\phi_i(t), i=1,\ldots,4$, are smooth real functions. When the
exponent in $t$ is a fraction, the right hand side should be set
to $0$.
\end{lem}

This sequence of lemmas deals with the general situation of a
singular orbit of codimension 2. We now specialize to our situation
with $G=\S^3\times\S^3$. Here we have, in terms of the basis
$X_i,Y_i$ of $\fg$, irreducible modules $\fm_1=\{X_2,X_3\}$ and
$\fm_2=\{Y_2,Y_3\}$, a trivial module $\fm_0$ spanned by
$W=-qX_1+pY_1$ and $\fp=\fk$ spanned by $Z=pX_1+qY_1$.

Applying \lref{slice} and \lref{trivial} (notice that $H$ acts the
same on $\fp$ and $\fm_0$) we get:

\begin{align*}
|Z^*|^2&=p^2\,f_1+q^2\,g_1+2pq\, h_1=
k^2 t^2+t^4\,\phi_1(t^2)\\
\ml Z^*,W^*\mr&=-p q\,f_1+pq\,g_1+(p^2-q^2)\,h_1=
t^2\phi_2(t^2)\\
|W^*|^2&=q^2\,f_1+p^2\,g_1-2pq\, h_1=\phi_3(t^2)
\end{align*}
This says in particular that $f_1,g_1,h_1$ must be even. The
equations for the values of the functions and their first and second
derivative at $t=0$ can be solved and give rise to the conditions in
\tref{gen smooth} (a).

On $\fm_1$ the isotropy group $\Ko$ acts by rotation $R(2p\theta)$
and on $\fm_2$  by  $R(2q\theta)$. Furthermore, the modules $\fm_1$
and $\fm_2$ are equivalent to each other under the action of $H$.
\tref{gen smooth} (b) then follows by applying \lref{onemod} and
\lref{twomod}. Notice that in our situation
$g_{12}=h_{12}=h_{21}=0$, as required by $H$ invariance of the
metric. This finishes the Proof of \tref{gen smooth}.
\end{proof}

\bigskip

\begin{center}
 \emph{ Curvature tensor of metrics on $P_k$ and $Q_k$}
 \end{center}

\bigskip

The following gives the formula for  the curvature tensor of a
general \coo\ metric on $P_k$ or $Q_k$ (and $R$).

\begin{thm}\label{gencurv}
A cohomogeneity one metric defined by $(f_i,g_i,h_i)$ with
$D_i=f_ig_i-h_i^2$, has the following components of the curvature
tensor, all others being $0$.
\begin{eqnarray*}
&R_{X_iY_iX_iY_i}=&-\frac 14 (f_i'g_i'-{h_i'}^2).\\
&R_{X_iY_iX_jY_j}=&-h_k- \frac 1{D_k} \left\{h_ih_j(f_k+g_k)+
h_k(f_ig_j+f_jg_i)-h_k(D_i+D_j)\right\}.\\
&R_{X_iX_jX_iX_j}=&2f_i+2f_j-3 f_k-\frac 14 f_i'f_j'
+\frac1{D_k}g_k(f_i-f_j)^2.\\
&R_{X_iX_jX_iY_j}=&h_j-\frac 14 f_i'h_j'-
\frac1{D_k}(f_i-f_j)(h_jg_k+h_ih_k).\\
&R_{X_iX_jY_iY_j}=&-2h_k-\frac 14 h_i'h_j'-
\frac 1{D_k}\left\{h_ih_j(f_k+g_k)+h_k(h_i^2+h_j^2)\right\}.\\
&R_{X_iY_jX_iY_j}=&-\frac 14 f_i'g_j'+
\frac 1{D_k}\left\{h_i^2f_k+h_j^2g_k+2h_ih_jh_k\right\}.\\
&R_{X_iY_jX_jY_i}=&h_k+\frac 14 h_i'h_j'+
\frac1{D_k} h_k(f_i-f_j)(g_i-g_j).\\
&R_{Y_iY_jX_iY_j}=&h_i-\frac 14 h_i'g_j'+
\frac1{D_k}(h_jh_k+h_if_k)(g_i-g_j).\\
&R_{Y_iY_jY_iY_j}=&2 g_i+2 g_j-3 g_k-
\frac 14 g_i'g_j'+\frac1{D_k}f_k(g_i-g_j)^2.\\
& R_{X_iX_jX_kT}=&\frac 12 f_i'+\frac 12 f_j'-f_k'
+\frac 1{2D_i}(f_j-f_k)(g_if_i'-h_ih_i')-
\frac 1{2D_j}(f_i-f_k)(g_jf_j'-h_jh_j').\\
& R_{X_iX_jY_kT}=&-h_k'+\frac 1{2D_i}\left\{h_j(f_ih_i'-h_if_i')-
h_k(g_if_i'-h_ih_i')\right\}\\
&&\hspace{150pt} - \frac 1{2D_j}\left\{h_i(f_jh_j'-h_jf_j')-
h_k(g_jf_j'-h_jh_j')\right\}.\\
& R_{X_iY_jX_kT}=&\frac 12 h_j'
+\frac 1{2D_i}\left\{h_j(g_if_i'-h_ih_i')-
h_k(f_ih_i'-h_if_i')\right\}-
\frac 1{2D_j}(f_i-f_k)(g_jh_j'-h_jg_j')\\
& R_{X_iY_jY_kT}=&\frac 12 h_i'
 + \frac 1{2D_i}(g_j-g_k)(f_ih_i'-h_if_i')-
\frac 1{2D_j} \left\{h_i(f_jg_j'-h_jh_j')-h_k(g_jh_j'-
h_jg_j')\right\}\\
& R_{Y_iY_jX_kT}=&-h_k'+\frac 1{2D_i} \left\{h_j(g_ih_i'-h_ig_i')-
h_k(f_ig_i'-h_ih_i')\right\} \\
&&\hspace{150pt} -\frac 1{2D_j}
\left\{h_i(g_jh_j'-h_jg_j')-
h_k(f_jg_j'-h_jh_j')\right\} .\\
& R_{Y_iY_jY_kT}=&\frac 12 g_i'+\frac 12 g_j'-g_k'
+  \frac 1{2D_i}(g_j-g_k)(f_ig_i'-h_ih_i')-
\frac 1{2D_j} (g_i-g_k)(f_jg_j'-h_jh_j')\\
&R_{X_iTX_iT}=&-\frac 12 f_i''+
\frac 1{4D_i}\left\{g_i{f_i'}^2+f_i{h_i'}^2-2h_if_i'h_i'\right\}.\\
&R_{X_iTY_iT}=&-\frac 12 h_i''+
\frac 1{4D_i}\left\{g_if_i'h_i'+
f_ih_i'g_i'-h_if_i'g_i'-h_i{h_i'}^2\right\}.\\
&R_{Y_iTY_iT}=&-\frac 12 g_i''+
\frac 1{4D_i}\left\{g_i{h_i'}^2+f_i{g_i'}^2-2h_ih_i'g_i'\right\}.\\
\end{eqnarray*}
with $(i,j,k)$ is a cyclic permutation of $(1,2,3)$.
\end{thm}

\begin{proof}
We will use the following curvature formulas for a cohomogeneity one
metric (see \cite{GZ2}):

\begin{align}
g(R(X,Y)Z,W) = &- \tfrac{1}{2}Q(B_-(X,Y),[Z,W]) - \tfrac{1}{2}Q([X,Y],B_-(Z,W))
\notag \\
&+\tfrac{1}{2}Q(P[X,Y]_{\fn},[Z,W]_{\fn}) +
\tfrac{1}{4}Q(P[X,Z]_{\fn},[Y,W]_{\fn})\notag \\
&-\tfrac{1}{4}Q(P[X,W]_{\fn},[Y,Z]_{\fn})\notag \\
& + Q(B_+(X,Z),P^{-1}B_+(Y,W)) - Q(B_+(X,W),P^{-1}B_+(Y,Z)) \notag \\
& + \tfrac{1}{4}Q(P'X,Z)Q(P'Y,W) - \tfrac{1}{4}Q(P'X,W)Q(P'Y,Z) \notag \\
g(R(X,Y)Z,T) = \: &\tfrac{1}{2} Q([X,Y],P'Z) - \tfrac{1}{4}
Q([Z,X],P'Y)
  - \tfrac{1}{4}Q([Y,Z],P'X) \notag \\
&- \tfrac{1}{2}Q(P'X,P^{-1}B_+(Y,Z)) +
\tfrac{1}{2}Q(P'Y,P^{-1}B_+(Z,X))\notag \\
g(R(X,T)T,Y) =\: &Q((-\tfrac{1}{2}P'' +
\tfrac{1}{4}P'P^{-1}P')X,Y)\notag
\end{align}

\no where $P$ defines the metric via $g(X^*,Y^*) = Q(P(X),Y)$ and
$B_{\pm}(X,Y) = \frac{1}{2}([X,PY] \mp [PX,Y])$.

For our metrics we have
$$
PX_i = f_iX_i + h_iY_i\ , \ PY_i = h_iX_i + g_iY_i
$$
and thus
$$
P^{-1}X_i = \frac{1}{D_i}(g_iX_i - h_iY_i) \ , \
P^{-1}Y_i = \frac{1}{D_i}(-h_iX_i + fY_i)
$$

\no  Since  $[X_i,X_j]=2 X_k \ ,\ [Y_i,Y_j]=2 Y_k $, one has
\begin{align*}
B_{\pm}(X_i,X_i) &=  B_{\pm}(Y_i,Y_i) = B_{\pm}(X_i,Y_i) = 0\\
B_{\pm}(X_i,X_j) &= (f_j \mp f_i)X_k\ , \
B_{\pm}(Y_i,Y_j) = (g_j \mp g_i)Y_k \ , \
B_{\pm}(X_i,Y_j) = (h_jX_k \mp h_iY_k)
\end{align*}
with $(i,j,k)$ cyclic. The formulas for the curvature tensor now
easily follow by substituting.
\end{proof}

Using these formulas, one also easily derives the curvature formulas
for a connection metric in \tref{curv conn}. We illustrate the
procedure in one particular case, the others being similar.

\bigskip

{\it Proof of \tref{curv conn}:}\ \ We will show that $\ml R( X_i^*
,\bar{Z}_j)\bar{Z}_k,T \mr=\e B_{ij}$.

 Notice that since $V_i=Y_i^*-h_iX_i^*=v_i\bar{Z}_i$  are not action fields, \cite{GZ2} can not be applied
directly. By expansion, we have:
\begin{align*}
 v_jv_k &\ml R( X_i^* ,\bar{Z}_j)\bar{Z}_k,T \mr=\ml R( X_i^* ,V_j)V_k,T \mr
= \langle R( X_i^* ,Y_j^*-h_j X_j^*)(Y_k^*-
h_kX_k^*),T \rangle \\
&=  \langle R( X_i^* ,Y_j^*)Y_k^*,T \rangle  -
h_j \langle R( X_i^*,X_j^*)Y_k^*,T \rangle -
h_k \langle R( X_i^* ,Y_j^*)X_k^*),T \rangle \\
&\quad +h_jh_k
\langle R(  X_i^*  ,X_j^* )X_k^* ,T \rangle
\end{align*}
We now use the curvature formulas from \tref{gencurv}, where we
replace $f_i$ by $\e$, $h_i$ by $\e h_i$ and $g_i$ by $v_i^2+\e
h_i^2$, as discussed in \eqref{scaling} for a scaled connection
metric. We thus have
\begin{align*}
&\langle R(  X_i^*  , Y_j^* )Y_k^* ,T \rangle=\frac{\epsilon}{2}\left(
\frac{v_i^2+v_k^2-v_j^2}{v_i^2}h_i' +h_k h_j' -2(h_i+h_jh_k)\frac{
v_j'}{v_j}\right)\\
&\hspace{100pt} +\frac{\epsilon^2}{2}\left(\frac {(h_k^2-h_j^2) h_i'}{ v_i^2}- \frac
{h_j h_j' (h_i+h_jh_k) }{ v_j^2}\right)\\
 &\langle R(  X_i^*  ,
X_j^* )Y_k^* ,T \rangle=-\epsilon h_k'-\frac{\epsilon^2}{2}\left(\frac{ h_i'}{
v_i^2}(h_j+h_ih_k) +\frac{h_j'}{ v_j^2}(h_i+h_jh_k)\right)\\
&\langle R(  X_i^*  , Y_j^* )X_k^* ,T \rangle=\epsilon \frac {h_j'}{2}
+\frac{\epsilon^2}{2} \frac{ h_i'}{  v_i^2}(h_k+h_ih_j)\\
&\langle R(  X_i^*  ,X_j^* )X_k^* ,T \rangle=0
\end{align*}

Combining these:

\begin{eqnarray*}
 v_jv_k \ml R( X_i^* ,\bar{Z}_j)\bar{Z}_k,T \mr&=&\frac{\epsilon}{2}\left(
\frac{v_i^2+v_k^2-v_j^2}{v_i^2}h_i'
+ h_k h_j'- 2(h_i+h_jh_k)\frac{
v_j'}{v_j}+2h_j h_k'-h_kh_j'\right) \\
&+&\frac{\epsilon^2}{2}\left(\frac {(h_k^2-h_j^2) h_i'}{ v_i^2}- \frac
{h_j h_j' (h_i+h_jh_k) }{ v_j^2}\right. \\
&&\hspace{15pt}+\left. h_j\frac{ h_i'}{
v_i^2}(h_j+h_ih_k) +h_j\frac{h_j'}{ v_j^2}(h_i+h_jh_k)
 -h_k\frac{ h_i'}{  v_i^2}(h_k+h_ih_j) \right)
\end{eqnarray*}
and thus
$$v_jv_k\ml R( X_i^* ,\bar{Z}_j)\bar{Z}_k,T \mr=
 \e ( h_jh_k'
 +\frac{v_k^2+v_i^2-v_j^2}{2v_i^2}h_i'-
\frac{v_j'}{v_j}(h_i+h_jh_k))=\e v_jv_k B_{ij}$$ which shows that
$\ml R( X_i^* ,\bar{Z}_j)\bar{Z}_k,T \mr=\e B_{ij}$.

\bigskip

\bigskip

We finally indicate how to prove the curvature formulas in
\eqref{curvbase} for the metric on the base. In this case the metric
is diagonal $PW^*_i = v_i^2 W^*_i$ and thus $B_{\pm}(W_i,W_i)  = 0$
and $B_{\pm}(W_i,W_j) =  (v_j^2 \mp v_i^2) W_k$, from which
\eqref{curvbase} easily follows as in the proof of \tref{gencurv}.

 \providecommand{\bysame}{\leavevmode\hbox
to3em{\hrulefill}\thinspace}

\end{document}